\documentclass[12pt]{amsart}
\usepackage{blindtext}
\usepackage[english]{babel} 
\usepackage[active]{srcltx}
\usepackage{subfigure} 
\usepackage{amssymb}
\usepackage{amsmath,amstext,amsthm,amsfonts,latexsym}
\usepackage[colorlinks=true,linkcolor=red,urlcolor=black]{hyperref} 
\usepackage{lineno} 
\usepackage{verbatim}

\usepackage{soul} 
\usepackage{graphicx}

\theoremstyle{plain}

\newtheorem{maintheorem}{Theorem}
\newtheorem{mainlemma}{Lemma}
\newtheorem{maincor}[maintheorem]{Corollary}

\newtheorem{teo}{Theorem}[section]
\newtheorem{prop}[teo]{Proposition}
\newtheorem{lema}[teo]{Lemma}

\makeatletter
\def\referencia#1#2{\begingroup
#2%
\def\@currentlabel{#2}%
\phantomsection\label{#1}\endgroup
}
\makeatother

\newcommand{\dist}{\mathrm{dist }}
\newcommand{\supp}{\mathrm{supp\, }}
\renewcommand{\S}{\mathbb{S}^1}
\newcommand{\Td}{\mathbb{T}^d}
\newcommand{\ms}{\mathrm{Leb}_{\mathbb{S}^1}}
\newcommand{\md}{\mathrm{Leb}_{\mathbb{T}^d}}
\newcommand{\m}{\mathrm{Leb}}
\newcommand{\hK}{\hat{K}}

\newcommand{\hS}{\hat{S}}
\newcommand{\hE}{\hat{E}}
\newcommand{\hpi}{\hat{\pi}}
\newcommand{\htheta}{\hat{\theta}}
\newcommand{\halpha}{\hat{\alpha}}
\newcommand{\hbeta}{\hat{\beta}}
\newcommand{\hgamma}{\hat{\gamma}}
\newcommand{\hmu}{\hat{\mu}}
\newcommand{\hnu}{\hat{\nu}}

\newcommand{\htop}{\mathrm{h_{top}}}
\newcommand{\Ptop}{\mathrm{P_{top}}}

\title{Measures maximizing the entropy for Kan endomorphisms}

\author{B\'arbara N\'u\~nez-Madariaga, 
Sebasti\'an A. Ram\'irez \& 
Carlos H. V\'asquez}
\address{Pontificia Universidad Cat\'olica de Valpara\'{\i}so, Blanco Viel 596, Cerro Bar\'on, Valpara\'{\i}so-Chile.}
\email{barbara.nunez@pucv.cl, sebastian.ramirez@pucv.cl, carlos.vasquez@pucv.cl}
\thanks{All the authors were partially supported by Proyecto Fondecyt 1210168. 
B.N.M. was partially supported by the National Agency for Research and Development (ANID) / Scholarship Program / Doctorado Becas Chile / 2021 - 21210374.
S.R. was partially supported by the National Agency for Research and Development (ANID) / Scholarship Program / Doctorado Becas Chile / 2018 - 21181113.}

\subjclass{Primary:37C40, 37D30, 37D35.}

\keywords{Partial hyperbolicity, Measures maximizing entropy, Lyapunov exponents}

\date{\today}

\begin{document}

\begin{abstract} In 1994, Ittai Kan provided the first example of maps with intermingled basins. The Kan example corresponds to a partially hyperbolic endomorphism defined on a surface, with the boundary exhibiting two intermingled hyperbolic physical measures. Both measures are supported on the boundary, and they also maximize the topological entropy. In this work, we prove the existence of a third hyperbolic measure supported in the interior of the cylinder that maximizes the entropy. We also prove this statement for a larger class of invariant measures of large class maps including perturbations of the Kan example.

\end{abstract}

\maketitle


\section{Introduction: The Kan example}\label{sec:Kan}

In 1994, Ittai Kan \cite{Kan:1994kw} provided the first examples of maps with intermingled basins. More precisely, he considered the map $K: M\to M$ defined on the cylinder $M=\mathbb{S}^1\times I$, $I=[0,1]$ by the explicit expression
\begin{equation}\label{eq:Kanoriginal}
K(\theta,t)=(3\theta \;(\hspace{-.4cm}\mod \mathbb{Z}), t+\cos(2\pi\theta)\left(\frac{t}{32}\right)(1-t)),
\end{equation}
where $(\theta,t)\in[0,1)\times I$.

Recall that the \textit{basin} of a $K$-invariant measure $\mu$ is the set
$$\mathcal B(\mu)=\{x\in M\::\: \lim_{n\to\infty}\frac{1}{n} \sum_{j=0}^{n-1}\psi(K^j(x))=\int \psi\,d\mu,\mbox{ for every } \psi\in C^0(M,\mathbb{R})\}.$$
A $K$-invariant probability measure $\mu$ is \textit{physical} if $\m(\mathcal B(\mu))>0$. Physical measures are relevant because they account for the statistics associated with the dynamics for a number of relevant points with respect to the intrinsic measure (Lebesgue measure) of the ambient space, even if it is not invariant. Two (physical) measures $\mu_0$ and $\mu_1$ are \textit{intermingled} (with respect to the Lebesgue measure in $M$) if for every open set $U\subseteq M$, we have
\begin{equation*}\label{eq:intermcondleb}
\m(\mathcal B(\mu_0)\cap U)>0\quad\mbox{ and }\quad \m(\mathcal B(\mu_1)\cap U)>0.
\end{equation*}
Kan proved that the measures $\mu_0=\ms\times \delta_0$ and $\mu_1=\ms\times\delta_1$ are physical, that their basins are intermingled and that their union covers $\m_M$-a.e. the cylinder $M=\mathbb{S}^1\times I$. Moreover, this phenomenon is robust under small perturbations \cite{IKS08}. 

Research on the problem of establishing conditions to guarantee the existence of physical measures has achieved great developments in the last 20 years, and the interested reader can refer to \cite{V1999, Y2002} and the references therein for an introductory reading. In the case of the Kan example, it corresponds to a mostly contracting partially hyperbolic system. Mostly contracting systems were introduced by Bonatti and Viana in \cite{BV2000}. These authors proved that such systems have only finitely many physical measures, and the union of their basins is a total Lebesgue measure. Some recent developments about mostly contracting systems can be found in \cite{A2010, VY2013, DVY2016}. 

The magic in the Kan's example occurs due two key facts:

\begin{itemize}
\item The Lebesgue measure is an invariant measure in each boundary and boundaries have negative central Lyapunov exponents, and
\item The fiber maps at $\theta = 0$ and $\theta = 1/2$ form a (heterodimensional) cycle between the boundaries.
\end{itemize} 

Such are the ingredients collected in \cite[chapter 11]{BDV} to define the \textit{Kan-like maps} as follows. 
We consider the more general (noninvertible) maps $K: M\to M$, being a local diffeomorphism defined on $M=\Td\times I$ by
\begin{equation}\label{eq:Kanmap}
K(\theta,t)=(E(\theta),\varphi(\theta, t)),
\end{equation}
where 
\begin{enumerate}
\item[\referencia{KE}{[KE]}] $E: \Td\to \Td$ is a $C^r$, $r\geq 1$, expanding map having two different fixed points $p, q\in \Td$. 
\end{enumerate}

Recall that a $C^1$-transformation $ f: M \to M $ defined on a compact connected Riemannian manifold $M$ is \textit{expanding} if there are $\lambda> 1$ and a metric on $ M $ such that $ \|Df (x) v \| \geq \lambda \| v \|,$ for all $ x \in M $ and all $ v \in T_x M $.

We also assume that $\varphi:M\to [0,1]$ is $C^r$, $r\geq 1$, satisfying the following conditions: 
\begin{enumerate}
\item[\referencia{K1}{[K1]}]For every $\theta\in \Td$ we have $\varphi(\theta,0)=0$ and $\varphi(\theta,1)=1$. 
\item[\referencia{K2}{[K2]}] For every $(\theta,t)\in M$, 
\begin{equation*}\label{eq:ph1}
|\partial_t\varphi(\theta,t))|<\frac{1}{2} \|DE (\theta)\|
\end{equation*}
\item[\referencia{K3}{[K3]}] The map $\varphi(p,\cdot):[0,1]\to[0,1]$ has exactly two fixed points: a hyperbolic source at $t=1$ and a hyperbolic sink in $t=0$. Analogously, the map $\varphi(q,\cdot):[0,1]\to[0,1]$ has exactly two fixed points: a hyperbolic sink at $t=1$ and a hyperbolic source in $t=0$.%
\end{enumerate}

Condition \ref{K1} means that $K$ preserves the boundary. Condition \ref{K2} states that the map $K$ is partially hyperbolic. Condition \ref{K3} allows relating the dynamics along the boundary through a (heteroclinical) cycle. We call such a local diffeomorphism $K$ given by \eqref{eq:Kanmap} satisfying conditions \ref{KE}, \ref{K1}, \ref{K2} and \ref{K3} a \textit{Kan-like map}. 

Kan's example \eqref{eq:Kanoriginal} belongs to the interesting particular class of Kan-like maps in which $E$ is an affine expanding map on $\S$; that is, $k\geq 3$:
$$E(\theta)= k\cdot\theta \:(\textrm{mod}\: \mathbb{Z}),\quad \theta\in \S,$$ 
for fixed $k\in\mathbb{Z}$. If $E$ is a linear expanding map, we state that $K$ is an \textit{affine Kan-like map}.

The next theorem provides a complete description of the physical measures for Kan-like maps.

\begin{teo}\label{teo:Kan}
Assume that $K$ is a $C^2$ Kan-like map and
\begin{equation}\label{eq:condK4phys}
\int\log |\partial_t \varphi (\theta,0)|\: d \md(\theta) <0 \mbox{ and }\int\log |\partial_t \varphi (\theta,1)|\: d\md(\theta)<0
\end{equation}
Then, there exist exactly two physical measures, $\mu_0$ and $\mu_1$, supported on the boundaries $\Td\times\{0\}$ and $\Td\times\{1\}$, respectively, such that they are intermingled and the union of their basins covers $\md$ at almost every point on the whole manifold.
\end{teo}

The reader can find a complete proof of this statement in \cite[Proposition 11.1]{BDV} (see also \cite{BM2008}). The proof of the Theorem~\ref{teo:Kan} above proceeds as follows. Recall that since $E$ is a $C^2$ expanding map, there exists a unique $E$-invariant probability measure $m_0$ which is absolutely continuous with respect to the Lebesgue measure on the torus. The measures in the theorem above are defined by $\mu_i=m_0\times\delta_i$, $i=0,1$. Condition \eqref{eq:condK4phys} refers to the fact that central Lyapunov exponents at the boundary points are generically (with respect to the Lebesgue measure) negative, and this implies (using Pesin's theory and Hopf's argument) that the measures are physical. 

Condition \ref{K3} is essential, implying that the basins are intermingled and that their union covers the whole manifold except for a zero Lebesgue measure subset.

Since Kan's work, the intermingled phenomenon has been the subject of various studies. For instance, there are some new examples that have more than two physical measures, and their basins are intermingled (see \cite{AHKY1996,MW2005,BP2018,DVY2016}). Some accurate estimations for the Hausdorff dimension of the set that does not belong to the basins of physical measures were obtained by \cite{IKS08, KS2011}.
Keller \cite{K2017} observed the phenomenon of intermingled basins for a skew product dynamical system on a square with a piecewise expanding Markov base map and a negative Schwarzian fiber map, and he further quantified the degree of intermingledness. Recently, Gan and Shi \cite{GS2019} proved that every affine Kan-like map on $\S\times I$ is $C^2$-robustly topologically mixing, providing more evidence about the difference between the topological and measurable indecomposability (see also \cite{CGS2018}). 

Bonifant and Milnor \cite{BM2008} also studied affine Kan-like maps on $\S\times I$ under the hypothesis of the negative Schwarzian derivative on the fiber map, which requires more regularity of the application $K$. They obtained the result contained in Theorem~\ref{teo:Kan}, and they also observed the existence of a third measure supported in the interior of the cylinder such that it is projected (via push forward) on the Lebesgue measure in the boundary. Those results motivate our study regarding other relevant measures for $K$. 

We slightly change the approach in this work, considering as a \textit{reference measure} any Borel probability $m$ defined on $M$, not necessarily Lebesgue. Therefore, an $f$-invariant probability measure $\mu$ is an \textit {observable measure with respect to the reference measure} $m$ if $ \mathcal {B} (\mu) $ has a positive $m$-measure. In particular, when $m=\m$ in $M$, $\mu$ is an observable measure when it is a physical measure.

Let $f:M\to M$ be a continuous map defined on a compact metric space having invariant measures $\mu$ and $\nu$. We say that (the basins of) $\mu$ and $\nu$ are \textit{intermingled with respect to a Borelean measure} $m$ defined on $M$ if for every open set $U\subseteq M$, we have
\begin{equation}\label{eq:intobs}
m(\mathcal{B}(\mu)\cap U)>0\quad\mbox{ and }\quad m(\mathcal{B}(\nu)\cap U)>0.
\end{equation}
It follows from \eqref{eq:intobs} that $\mu$ and $\nu$ are observable measures with respect to $m$. 
Note that in the particular case when $M$ is a Riemannian manifold and $m=\m$, we recover the notion of physical measures whose basins are intermingled.

Our first result is the following.

\begin{maintheorem}\label{mainteo:KE1} 
Assume that $K$ is a Kan-like map and suppose that $\nu$ is any 
ergodic $E$-invariant probability measure on $\Td$ satisfying
\begin{equation}\label{eq:condK4im}
\int_{\Td} \log |\partial_t \varphi(\theta,j)| d\nu(\theta) < 0, \quad j=0,1. 
\end{equation}
Denote by $m= \nu \times \m_I$ the reference measure. Then, the following is true:
\begin{enumerate}
  \item[\referencia{TeoAi}{(i)}]{\rm{\textbf{Measures supported in the boundaries:}}} The measures $\mu_j=\nu \times \delta_j$ are  $K$-invariant and ergodic, have negative central Lyapunov exponents, are supported on $\Td \times \{j\}$, and satisfy the following properties :
    \begin{itemize}
        \item[(a)] The measures $\mu_j$, $j=0,1$, are observable measures with respect to $m$.
    
        \item[(b)] If $\supp \nu =\Td$, then the basins $\mathcal{B}(\mu_j)$ are intermingled with respect to $m$.

        \item[(c)] If the Jacobian $J_\nu E$ has bounded distortion\footnote{We refer the reader to subsection~\ref{ssec:expandingmaps} for details about the notions of Jacobian and bounded distortion.} in addition, then the union of their basins $\mathcal{B}(\mu_0)\cup \mathcal{B}(\mu_1)$ covers $m$-a.e. the  manifold $M$.
    \end{itemize}    
Moreover, there is no other invariant measure $\mu$ such that $\pi_*\mu=\nu$ with a negative center Lyapunov exponent. 

\end{enumerate} 

\begin{enumerate}
\item[\referencia{TeoAii}{(ii)}]{\rm{\textbf{Measures supported in the interior:}}}
Suppose that $\nu$ is any ergodic $E$-invariant probability measure such that $J_\nu E$ is H\"older continuous. Then,
\begin{enumerate}
        \item[(d)] There is a measurable, not continuous, map $\sigma: \Td \to I$ such that for $\nu$-a.e. $\theta\in \Td$ 
        $$W^s(\theta, 0)=\{\theta\} \times [0, \sigma(\theta)),\quad\mbox{ and }\quad W^s(\theta, 1)=\{\theta\} \times (\sigma(\theta), 1],$$ where $W^s (\theta, j)$ denotes the stable manifold.

       \item[(e)] There exists a unique ergodic $K$-invariant probability measure $\mu$ such that $\pi_*\mu=\nu$ and $\mu\ne\mu_j$, $j=0,1$. Such measure $\mu$ is defined by
$$\int_M\Phi\,d\mu=\int_{\Td}\Phi(\theta,\sigma(\theta))\,d\nu(\theta),$$
for every $\Phi:M\to \mathbb{R}$ continuous.
        \item[(f)] Moreover, $\mu$ satisfies:
        \begin{enumerate}
        \item The central Lyapunov exponent of $\mu$ is positive,
        \item $\mu$ is approximated by measures supported along periodic orbits in the interior of $M$, and
        \item if $K$ is $C^2$ then, $\supp\mu=M$.
            \end{enumerate}
    \end{enumerate}

\end{enumerate}
\end{maintheorem}

In general, the idea of the proof of Theorem~\ref{mainteo:KE1} of part \ref{TeoAi} is to consider $\mu_j=\nu\times\delta_j$, with $j=0,1$, and to observe that the assumption \eqref{eq:condK4im} means that $\mu_j$  has a negative central Lyapunov exponent. We can then follow the argument of Kan (see \cite[chapter 11]{BDV} and \cite{BV2000}). The main obstacle is the use of properties such as absolutely continuous holonomies, density points, or distortion in order to prove (c). In contrast, we strongly use the product structure, density points for Borel measures (from the Besicovich covering theorem), and bounded distortion of the Jacobian.
 
The proof of Theorem~\ref{mainteo:KE1} in part \ref{TeoAii} follows \textit{a la Bonifant-Milnor}. Obtaining the measure slides the mass to one end of the local stable leaf, such as the measure obtained in \cite{RRTU2012} for measures maximizing the entropy. Similar to that work, a key step in the proof of (g) is to establish Lemma~\ref{le:klemma}, which is an application of the invariance principle (see \cite{AV2010}), which must be reformulated to the noninvertible case. H\"older continuity of the Jacobian is the key in this last step. We want to emphasize that our methods work for $C^1$ maps. Nevertheless, the proof of item (f) part (iii) above follows the ideas of \cite{GS2019} where the Sternberg's result is invoked, and so we require $C^2$ regularity.

Recall that an invariant probability measure $\mu$ is a \textit{measure maximizing the entropy} if $ {\rm h}_\mu (K) = \htop (K)$ (see more details in Section~\ref{sec:preliminares}). Such measures are the equilibrium states for the constant potential $\varphi=0$. As an application of Theorem~\ref{mainteo:KE1}, we can recover the conclusion of Theorem~\ref{teo:Kan} for such a class of measures. Recall that if $E$ is an expanding map, then there exists a unique measure maximizing the entropy $\nu_0$ supported on the whole $\Td$ and having a constant Jacobian. 

\begin{maincor}\label{maincor:B} 
Assume that $K$ is a Kan-like map and $\nu_0$ is the supposed   unique measure maximizing the entropy of $E$. Assume that
\begin{equation}\label{eq:condK4mme}
\int\log |\partial_t \varphi (\theta,j)|\: d\nu_0(\theta)<0, \quad j=0,1;
\end{equation}
\noindent Then, there exist exactly three ergodic measures maximizing the entropy 
$\mu_0,\mu,\mu_1$. All of them are hyperbolic, and 
\begin{enumerate}
\item[\referencia{CorBi}{(i)}] $\mu_0$ and $\mu_1$ have negative central Lyapunov exponents, they are supported on the boundaries, each basin has positive weight with respect to $\nu_0\times \m_I$, and the union of their basin cover $\nu_0\times \m_I$ a.e. the whole manifold $M$.
Moreover, their basins are intermingled with respect to $\nu_0\times \m_I$. 
\item[\referencia{CorBii}{(ii)}] $\mu$ has a positive central Lyapunov exponent, $\mu$ is approximated by ergodic measures supported on periodic orbits of $K$ in the interior of $M$, and $\supp \mu=M$ if $K$ is $C^2$.
\end{enumerate}
\end{maincor}

The result above is similar to the conclusion obtained in \cite{RRTU2012} for accessible partially hyperbolic diffeomorphisms of three-dimensional manifolds having compact center leaves. We note that the result above can be extended to the existence of equilibrium states of H\"older continuous potentials which are constant on the fibers. It is also worth mentioning \cite{CT2019}, which includes the same construction although the setting is somewhat different. We also mention \cite{CP2020}, where the authors prove existence and uniqueness of equilibrium states for the Shub example on $\mathbb{T}^4$. As in our case, they consider potentials which are well projected along certain central foliation to the 2-torus where the dynamic is Anosov.

We can extend almost all of the conclusion of Theorem~\ref{mainteo:KE1} to a class of partially hyperbolic maps preserving the boundary and having a regular center foliation. More precisely, we consider $F: M\to M$, being a $C^r$-local diffeomorphism, $r\geq 1$, defined on $M=\Td\times I$ by
\begin{equation*}\label{eq:Fibmap}
F(\theta,t)=(E(\theta,t),\varphi(\theta, t)),
\end{equation*}
and $F$ satisfies the following conditions: 
\begin{enumerate}
\item[\referencia{F1}{[F1]}] $F$ \textit{preserves the boundary}: For every $\theta\in \Td$, we have $\varphi(\theta,0)=0$ and $\varphi(\theta,1)=1$. 
\item[\ref{F2}] $F$ is a \textit{partially hyperbolic endomorphism}. 
\item[\ref{F3}] $F$ \textit{preserves an invariant center foliation $\mathcal{F}^c$.} 
\item[\referencia{F4}{[F4]}] $F$ \textit{relates the dynamics along the boundary through a heteroclinical cycle}: The map $E(\cdot,0)$ has two different fixed points $p, q\in \Td$. Moreover, the map $\varphi_p^c:I\to I$, defined by $\varphi_p(t)=\varphi(p,t)$, has exactly two fixed points: a hyperbolic source at $t=1$, and a hyperbolic sink at $t=0$. Analogously, the map $\varphi_q^c:I\to I$ has exactly two fixed points: a hyperbolic sink at $t=1$, and a hyperbolic source at $t=0$.

\end{enumerate}
Precise statements of conditions \ref{F2} and \ref{F3} can be found in Subsection~\ref{ssec:hiperbparcial} and Subsection~\ref{ssec:foliaciones}, respectively.


It follows from \ref{F1} and \ref{F2} that for $j=0,1$, $E_j:\Td\to\Td$ defined by $E_j(\theta)=E(\theta,j)$ is a $C^2$ expanding map. 


Note that we have a natural identification among the (ergodic) $E_j$-invariant probability measures $\nu_j$ on $\Td$ and the (ergodic) $F$-invariant probability measures $\mu_j$ supported on the boundary $\Td\times\{j\}$ given by $\mu_j=\nu_j\times \delta_j$. In the follow,  denote by $\nu_j$ the invariant probability measure maximizing the entropy corresponding to $E_j$. 


Furthermore,  the existence of a continuous \textit{natural holonomy along the central leaves} $\pi_{j, j^*}:\Td\to \Td$ is deduced from \ref{F3}: If $j,j^*\in\{0,1\}$ such that $j+j^*=1$, we define  $\pi_{j,j*}(\theta) \in\Td$ such that $\mathcal F^c(\theta,j)=\mathcal F^c(\pi_{j,j*}(\theta),j^*)$.  

Therefore, we have that $\pi_{j,j^*}$ is a topological conjugacy between $E_j$ and $E_{j^*}$, and thus it relates the unique measures of maximal entropy for both maps. Thus, if $\mu_j=\nu_j\times\delta_j$, the following relation is satisfied:
\begin{equation}\label{eq:imagenmedintro}
\mu_{j^*}=(\pi_{j,j^*})_*\mu_j.
\end{equation}

Simirlanly, we define  the \textit{natural projection along the central leaves} $\pi_j:M\to \Td$ as $\pi_j(\theta,t)\in\Td$ such that $\mathcal F^c(\theta,t)=\mathcal F^c(\pi_j(\theta,t),j)$.



\begin{enumerate}
\item[\referencia{F5}{[F5]}] The $F$-invariant measures supported on the boundary have \textit{negative central Lyapunov exponents}: Assume that
\begin{equation}\label{eq:condF5mme}
\int\log |DF|E^c|\: d\,\mu_j(\theta)<0, \quad j=0,1.
\end{equation}
\end{enumerate}

Therefore, we can prove the following general version of Theorem~\ref{mainteo:KE1}.

\begin{maintheorem}\label{mainteo:F} 
Let $F:M\to M$ be a $C^r$-local diffeomorphism, $r\geq 1$, satisfying \ref{F1}-\ref{F5}. 
Assume that $\nu_j$ is any ergodic $E_j$-invariant probability measure with $\supp \nu_j=\Td$ and $J_{\nu_j} E_j$ has bounded distortion. Suppose that $\mu_j=\nu_j\times\delta_j$ satisfy  the relation \eqref{eq:imagenmedintro}. Then, there exists a reference measure $m$ on $M$ such that $(\pi_j)_*m=\nu_j$ and its conditional measures on center leaves are equivalent to Lebesgue.  Besides,
   \begin{itemize}
    \item[(i)]   the measures $\mu_j$, $j=0,1$, are observable with respect to $m$,  their basins are intermingled with respect to $m$, and the union of their basins $\mathcal{B}(\mu_0)\cup \mathcal{B}(\mu_1)$ covers $m$-a.e. the whole manifold $M$. Moreover, there is no other $F$-invariant measure $\mu$ on $M$ such that $(\pi _j )_* \mu=\nu_j $ with a negative center Lyapunov exponent. 

        \item[(ii)] There is a unique ergodic $F$-invariant probability measure $\mu$ such that $(\pi_j)_*\mu=\nu_j$ and $\mu\ne\mu_j$, $j=0,1$. Such measure $\mu$ has a nonnegative central Lyapunov exponent, where $\mu$ is approximated by measures supported along periodic orbits in the interior of $M$. If $F$ is $C^2$ and the leaves of $\mathcal{F}^c$ are $C^2$, then $\supp\mu=M$.      

\end{itemize}

\end{maintheorem}

We point out that in this case, we do not have a cocycle associated with the inverse limit, so we cannot apply the invariance principle and then, we cannot discard the possibility of a zero central Lyapunov exponent in (ii). 

Kan in \cite{Kan:1994kw} claimed that Theorem~\ref{teo:Kan} would apply not only to skew product transformations like \eqref{eq:Kanoriginal} but also to an open set of perturbations preserving the boundary; however, he offered a very brief discussion of the complications at the end of the paper. This is proven formally in \cite{IKS08}.

Note that condition \ref{F2} is $C^r$-open in the set ${\rm End}^r(M,\partial M)$ of $C^r$-local diffeomorphisms on $M$ preserving the boundary, $r\geq 1$. It follows from the classical normal hyperbolicity theory \cite{HPS1977}  that small perturbations of the Kan-like maps satisfy \ref{F3} (see \cite{GS2019, IKS08}). Moreover, condition \ref{F5} is also open: this can be verified directly for Kan's original example \eqref{eq:Kanoriginal} if we replace the first coordinate by any linear expanding map on the torus, where $\md$ is the (unique) measure maximizing the entropy. If $E$ is a $C^r$-perturbation of a linear expanding map on the torus, $r\geq1$, then $E$ has a measure maximizing the entropy $\nu_E$ which is close (in the weak* topology) to $\md$. Therefore, if \eqref{eq:condK4mme} is negative for Lebesgue, then the integral remains negative when Lebesgue is replaced by $\nu_E$ (see \cite{VV2010} and the discussion in subsection~\ref{ssec:condnegLyap} for more details). Therefore, we have the following result for perturbations of  Kan-like maps when we consider measures maximizing the entropy.

\begin{maintheorem}\label{mainteo:C} Let $K:M\to M$ be a Kan-like map satisfying \ref{K1}-\ref{K3}. Let $\nu_0$ be the unique measure of the maximal entropy of $E$. Assume that
$$
\int\log |\partial_t \varphi (\theta,j)|\: d\nu_0(\theta)<0, \quad j=0,1;
$$
Then, there exists a $C^r$-neighborhood $\mathcal U\subseteq {\rm End}^r(M,\partial M)$, with $K\in \mathcal U$ such that for every $F\in\mathcal U$ there exist exactly three measures maximizing the entropy of $F$, namely, $\mu_0^F,\mu^F,\mu_1^F$, all of which are hyperbolic such that they verify the following:

\begin{enumerate}
\item $\mu_0^F$ and $\mu_1^F$ have negative central Lyapunov exponents, they are supported on the boundaries, and there is a reference measure $m^F$ defined on $M$ such that each basin of $\mu_j$ has a positive weight with respect to $m^F$ and the union $\mathcal B(\mu_0^F)\cup\mathcal B(\mu^F_1)$ covers $m^F$-a.e. the whole manifold $M$. Moreover, their basins are intermingled with respect to $m^F$. 
\item $\mu^F$ has a positive central Lyapunov exponent and it is approximated by ergodic measures supported on periodic orbits of $F$ in the interior of $M$.  Futhermore, 
$$\pi_* \mu_0^F=\pi_* \mu^F=\pi_* \mu_1^F=\pi_* m^F.$$
Moreover, if $K$ is  $2$-partially hyperbolic\footnote{See subsection~\ref{ssec:foliaciones} for a precise definition of $r$-partially hyperbolic, $r\geq 1$.}, then  $\mu^F$ is fully supported on $M$.
\item $\mu_j^F\to\mu_j^K$, $\mu^F\to\mu^K$, in the weak* topology when $F\to K$.

\end{enumerate}

\end{maintheorem}

This work is organized as follows. In Section~\ref{sec:preliminares}, we establish the basic notions and the set of necessary tools that we use throughout this work. Section~\ref{sec:mmei} is dedicated to the proof of Theorem~\ref{mainteo:KE1} part \ref{TeoAi}, and Section~\ref{sec:mmeii} is dedicated to the proof of Theorem~\ref{mainteo:KE1} part \ref{TeoAii}. A key step in the proof of the last part is to establish Lemma~\ref{le:klemma}, which is the application of the invariance principle that we developed in Section~\ref{sec:Klemma}. Finally, the highlights of the proof of Theorem~\ref{mainteo:F} are developed in Section~\ref{sec:ProofThF}. The proof of Theorem~\ref{mainteo:C} is given in Section~\ref{secc:applications}.


\section{Preliminaries}\label{sec:preliminares} 
\subsection{Topological pressure and equilibrium states}\label{ssec:entropy}

First, recall the definition of topological entropy. Let $(X,\dist)$ be a compact metric space and let $f:X\to X$ be a uniformly continuous map. For $x\in X$, $n\geq 1$ and $\epsilon>0$, the \textit{dynamical ball} is the set defined by
$$B(x,n,{\epsilon})=\big\{y\in X: \max_{0\leq i\leq n} \textrm{dist}(f^ix,f^iy)<\epsilon\big\}.$$
Fix a  compact set $Y \subseteq X$. A subset $\mathcal{E}\subseteq Y$ is called $(n,\epsilon)$-\emph{separated} if for $x\neq y\in \mathcal{E}$ there exists $1\leq i\leq n-1$ such that $\textrm{dist}(f^ix,f^iy)>\epsilon$. In other words, the points of $ \mathcal{E}$ are separated from each other by dynamical balls of length $n$ and radius $\epsilon$.

We denote by $s(n,\epsilon,Y)$ the cardinal of the maximum $(n,\epsilon)$-separated set in $Y$. The \textit{topological entropy of $f$ with respect to $Y$} is the limit
\begin{equation}\label{eq:topent}
\htop(f,Y):=\lim_{\epsilon\rightarrow 0}\limsup_{n\rightarrow\infty} \frac{1}{n}\log s(n,\epsilon,Y).
\end{equation}
Note that the limit \eqref{eq:topent} always exists (see \cite{Wal}). We often write $\htop(f)$ instead of $\htop(f,X)$.

Let $\Phi :X \to \mathbb{R}$ be a continuous function. We denote by 
$$S(\Phi,n,\epsilon)=\sup \lbrace \sum_{x \in \mathcal{E}} e^{\Phi_n(x)}: \mathcal{E} \text{ is an } (n,\epsilon)-\text{separated set for } M \rbrace ,$$ 
where $\Phi_n= \sum_{i=0} ^{n-1} \Phi \circ f^i$. The \textit{topological pressure} of $f$ with respect to the potential $\Phi$ is the limit 
\begin{equation}\label{eq:pressure}
    \Ptop (f,\Phi)=\lim_{\epsilon \to 0} \limsup_{n} \frac{1}{n}\log S(\Phi,n,\epsilon).
\end{equation}
Note that the limit \eqref{eq:pressure} always exists (see \cite[Section 9.1]{Wal}). An interesting particular case is to consider the potential $\Phi=0$, and then the pressure $\Ptop(f,0)=\htop(f)$.

Let $X$ and $Y$ be compact metric spaces, and let $f :X \to X$, $g: Y \to Y$ be continuous maps. Let us suppose that the maps $f$ and $g$ are semiconjugated by the map $\pi$: that is, $\pi: X\to Y$ is a surjective continuous map such that $\pi\circ f = g\circ \pi$. Then, Bowen \cite[Theorem 17]{B1971} has shown that 
\begin{equation*}\label{eq:Bowen}
\htop(f)\leq \htop(g)+\sup_{y\in Y}\htop(f,\pi^{-1}(y)).
\end{equation*}
Furthermore, Carvalho and Perez \cite[Proposition 5.1]{CP2020} proved that for every potential $\phi:Y \to \mathbb{R}$, we have
\begin{equation}\label{eq:CarvalhoPerez}
    \Ptop(f, \phi \circ \pi) \leq \Ptop (g,\phi ) + \sup \left\{ \int \htop(f, \pi^{-1}(y)) d\nu :\nu\in M^1(g) \right\},
\end{equation} 
where $M^1(f)$ denotes the set of $f$-invariant Borel probability measures on $Y$. We denote by $M^{erg}(f)$ the subset of the ergodic ones. 

Let $\mu\in M^1(f)$. Given $x\in X$, $n \geq 1$ and $\epsilon >0$, we introduce the following values, 
$${\rm h}_\mu^+(f,\epsilon, x)= \limsup_n - \frac{1}{n}\log \mu (B(x, n, \varepsilon)), $$
$${\rm h}_\mu^-(f,\epsilon, x)= \liminf_n - \frac{1}{n}\log \mu (B(x, n, \varepsilon)). $$
The Brin-Katok Theorem \cite{BK1983} ensures that for $\mu$-a.e. $x \in X$, the following limits exist and 
$$ {\rm h}_\mu(f,x): = \lim_{\epsilon \to 0} {\rm h}_\mu^+(f, \epsilon, x) = \lim_{\epsilon \to 0} {\rm h}_\mu^-(f, \epsilon, x).$$ 
Moreover, the function ${\rm h}_\mu(f,\cdot)$ is integrable and $f$-invariant. Therefore, we can define the \textit{measure-theoretic (metric) entropy} of $f$ with respect to the invariant measure $\mu$ as
$${\rm h}_\mu (f) := \int {\rm h}_\mu (f,x) d \mu (x).$$
Let $f:X \to X$ and $g:Y\to Y$ be continuous maps and suppose that they are semiconjugated by $\pi:X\to Y$. For $\mu\in M^1(f)$, we define the \textit{push forward} $\pi_*\mu\in M^1(g)$ by
$$\pi_*\mu(A):=\mu(\pi^{-1}(A)),$$
for every Borelean set $A\subseteq Y$. It follows directly from the definitions above that $\pi_*\mu$ is a $g$-invariant probability measure and
$${\rm h}_{\pi_*\mu}(g)\leq{\rm h}_\mu(f).$$
Similar to the Bowen formula, the Ledrappier-Walters formula \cite{LW1977} relates the entropies of two semiconjugated applications. More precisely,
\begin{equation}\label{eq:LW}
\sup\big\{h_\mu(f):\mu\in M^1(f), \pi_*\mu=\nu\big\}=h_\nu(g)+\int_Y \htop(f,\pi^{-1}(y))\,d\nu(y).
\end{equation}
The metric entropy and topological entropy are related through the variational principle \cite[Theorem 10.1]{OV2016}
\begin{equation}\label{eq:VP}
\htop (f)= \sup \big\{{\rm h}_\mu(f) \, : \, \mu \in M^1(f)\big\}.
\end{equation}
We state that an $f$-invariant measure $\mu$ \textit{maximizes the entropy} if its metric entropy realizes the supremum in \eqref{eq:VP}: that is, 
$$\htop (f)={\rm h}_\mu(f) .$$
Analogously, we can discuss the variational principle for pressure, \cite[Theorem 10.4.1]{OV2016}  
\begin{equation*}\label{eq:VP2}
\Ptop (f, \Phi)= \sup \big\{{\rm h}_\mu(f) + \int \Phi d\mu \, : \, \mu \in M^1(f)\big\}.
\end{equation*}
We state that an $f$-invariant measure $\mu$ is an \textit{equilibrium state} for the continuous potential $\Phi$ if 
$$\Ptop (f, \Phi)={\rm h}_\mu(f) +\int \Phi d\mu.$$
We denote by $M_\Phi(f)$ the set of ergodic equilibrium states for $\Phi$.  In particular, $M_\Phi^{erg}(f)$ denotes the set of ergodic $f$-invariant probability measures maximizing the entropy.

\subsection{Disintegration}\label{ssec:disintegration}
Throughout this section we will assume that $X$ is a separable, complete metric space and that $\mu$ is a probability measure defined on the Borelean subsets of $X$. 

\noindent Let $\mathcal{P}$ be a partition of $X$ in measurable subsets. We denote by $P: X\to \mathcal{P}$ the natural projection that associates to every $x\in X$ the element of the $P(x)$ containing $x$. In $\mathcal{P}$, we induce a measurable structure as follows. A subset $\mathcal{Q}\subseteq\mathcal{P}$ is measurable if $P^{-1}(\mathcal{Q})\subseteq X$ is Borelean. Now, we define the quotient measure $\hmu$ by
$$\hmu(\mathcal{Q})=\mu(P^{-1}(\mathcal{Q})),\quad \mbox{ for every } \mathcal{Q}\subseteq\mathcal{P} \mbox{ measurable }.$$
A \textit{disintegration} of $\mu$ relative to $\mathcal{P}$ is a family $\{ \mu_P \: : \: P \in \mathcal{P}\}$ of probability measures on $X$ satisfying
\begin{itemize}
\item[a)] $\mu_P (P)=1$ for $\hmu$-almost every $P \in \mathcal{P}$,
\item[b)] For every measurable subset $E \subseteq X$, the map $\mathcal{P} \to \mathbb{R}$ defined by $ P \to \mu_P(E)$ is measurable,
\item[c)]$\mu(E) = \int \mu_P(E) d \hmu (P)$, for every measurable subset $E \subseteq X$.
\end{itemize}

The family of probability measures $\{ \mu_P \: : \: P \in \mathcal{P}\}$ is called a family of \textit{conditional measures} of $\mu$ with respect to the partition $\mathcal{P}$.

Recall that a family $\mathcal{P}$ of Borelean subsets of $X$ is a \textit{measurable partition} if there exists a subset $X_0 \subseteq X$ of full measure and an increasing sequence $\mathcal{P}_1 \prec \mathcal{P}_2 \prec \cdots \prec \mathcal{P}_n \prec \cdots$ of countable partitions of $X_0$ (by a Borelean subset of $X_0$) such that 
$$\mathcal{P}=\bigvee_{n=1}^\infty \mathcal{P}_n.$$
\noindent Rokhlin's Theorem \cite{R1949} (see also \cite[Theorem 5.1.11]{OV2016}) establishes that if $\mathcal{P}$ is a measurable partition, then there exists the disintegration of $\mu$ with respect to $\mathcal{P}$.

A consequence of Rokhlin's Theorem is the ergodic decomposition theorem \cite[Theorem 5.1.3]{OV2016}: if $f:X \to X$ is measurable and $\mu$ is an $f$-invariant probability, then there exists a measurable subset $X_0 \subset X$, $\mu(X_0)=1$, a measurable partition $\mathcal{P}$ of $X_0$ and a family of conditional measures $\{ \mu_P \ : \ P \in \mathcal{P}\}$ such that it is a disintegration of $\mu$, where the conditional measures $\mu_P$ are $f$-invariant and ergodic. 


\subsection{Lyapunov exponents}\label{ssec:Lyapunov} 
Let $M$ be a compact Riemannian manifold of dimension $d$ and let $f:M\to M$ be a map of class $C^1$. Oseledets' Theorem \cite[Theorem 4.2]{V2014}  establishes that if $\mu$ is an $f$-invariant probability measure; then, for almost every $x\in M$, there exist $1\leq k=k(x)\leq d$ real numbers
$$\lambda_1(x)>\lambda_2(x)>\dots>\lambda_{k}(x),$$
\noindent and a decreasing sequence of subspaces
$$T_xM=V^1_x\supsetneq V^2_x\supsetneq\dots\supsetneq V^k_x\supsetneq \{0\}$$
\noindent such that for every $1\leq j\leq k$, the maps $x\to k(x)$, $x\to \lambda_j(x)$ and $x\to V^j_x$ are measurable and invariant (by $f$ or $Df$ as appropriate) and
$$\lambda_j(x)= \lim_{n \to \infty}\frac{1}{n} \log\Vert Df^n(x)v \Vert, \mbox{for every } v\in V^i_x\setminus V^{i+1}_x.$$
The numbers $\lambda_j(x)$ are called the \textit{Lyapunov exponents} of $f$ on $x\in M$. If $\mu$ is ergodic, then the values of $k(x)$ and $\lambda_j(x)$ as well as the dimensions of the spaces $V^j_x$ are constant a.e.

Let $\mu$ be $f$-invariant and ergodic. We denote by $\mathcal{R}(\mu)$ the full measure set provided by Oseledets' Theorem. For every $x \in \mathcal{R}(\mu)$, define
$$E^s_x = \bigcup_{\lambda_i(x) < 0 } V^{(i)}(x),$$
\noindent and
$$s:= \# \big\{\lambda_i(x)<0 : \, 1 \leq i \leq k(x) \big\}.$$
\noindent (Since $\mu$ is ergodic, we have that $s(x)=s$ is constant). Suppose that $s\geq 1$. Pesin's Theory \cite[Theorem V.6.5]{QXZ2009} guarantees that there exists $\epsilon(x)>0$ such that the set
$$
W^s_\epsilon(x):=\big\{ y \in M \: : \: \limsup_{n \to \infty} \frac{1}{n} \log\dist(f^n(y), f^n(x))<0,\,
\sup_{n \geq 0}\dist(f^n(y), f^n(x))\leq \epsilon (x)\big\}
$$
\noindent is an embedded disc on $M$ of dimension $\dim E^s_x$ that is the same class of differentiability as $ f $ and that satisfies $T_xW^s_\epsilon(x)=E^s_x$, and the set
$$W^s(x):= \big\{ y \in M \: : \: \limsup_{n \to \infty} \frac{1}{n} \log \dist(f^n(y), f^n(x))<0 \big\},$$
satisfies $f(W^s(x))=W^s(f(x))$ and
$$W^s(x)=\bigcup_{n\geq 0}f^{-n}(W^s_\epsilon(f^n(x)).$$
The set $W^s_\epsilon(x)$ is called the \textit{Pesin local stable manifold} at the point $x$, and the set $W^s(x)$ is called the \textit{Pesin's stable manifold} of $x$.


\subsection{Expanding maps}\label{ssec:expandingmaps}
For future reference, we summarize some of the main properties for the dynamics of the expanding map $E$ invoked in this work. The interested reader can find more details of the theory in \cite[chapters 11-12] {OV2016}. Along this subsection, $M$ will denote a compact connected Riemannian manifold.

Recall that a $C^1$-transformation $ f: M \to M $ defined on a compact connected Riemannian manifold $M$ is \textit{expanding} if there are $\lambda> 1$ and a metric on $ M $ such that $ \|Df (x) v \| \geq \lambda \| v \|,$ for all $ x \in M $ and all $ v \in T_x M $. It is clear from the definition that any map sufficiently close to an expanding one, relative to the $C^1$-topology, is still expanding.

We are interested in the case in which $M=\Td$ and $f=E$. A classical result of Michael Shub \cite{S1969} asserts that every expanding map on the torus $\Td$ is topologically conjugated to an expanding linear endomorphism: that is, there exists some $d\times d$ matrix $A$ with integer coefficients and a determinant different from zero such that all the eigenvalues $\lambda_1,...,\lambda_d$ of $A$ have modulo larger than 1. Therefore, $E$ conjugates to $E_A$ the endomorphisms induced by $A$ in $\Td=\mathbb R^d/\mathbb Z^d$. 

In particular, there exists an integer $k\geq 1$ such that every $\theta\in\Td$ has exactly $k$ preimages $\theta_1,\dots,\theta_k$. We call the integer $k\geq2$ the \textit{degree} of $E$. 

For an expanding local diffeomorphism $E$ on $\Td$, there exists $r_0>0$ such that the restriction of $E$ to each ball $B(\theta,r_0)$ of radius $r_0$ is injective, its image contains the closure of $B(E(\theta),r_0)$ and
\begin{equation}\label{eq:expandineq1}
\textrm{dist}(E(\theta_1),E(\theta_2))\geq\lambda\, \dist(\theta_1,\theta_2), \qquad\mbox{ for every }\theta_1,\theta_2\in B(\theta,r_0).
\end{equation}
Thus, the restriction to $B(\theta,r_0)\cap E^{-1}(B(E(\theta),r_0))$ is a homeomorphism onto $B(E(\theta),r_0)$. We denote its inverse by $E^{-1}_\theta :B(E(\theta),r_0)\to B(\theta,r_0)$ and call it the \textit{inverse branch} of $E$ at $\theta$. It is clear that $E^{-1}_\theta(E (\theta)) = \theta$ and $E\circ E^{-1}_\theta=id$. Condition \eqref{eq:expandineq1} implies that $E^{-1}_\theta$ is a $\lambda^{-1}$-contraction: 
\begin{equation*}\label{eq:expandineq2}
\dist(E^{-1}_\theta(\tilde\theta_1),E^{-1}_\theta(\tilde\theta_2))\leq\lambda^{-1}\dist(\tilde\theta_1,\tilde\theta_2) \mbox{ for every } \tilde\theta_1,\tilde\theta_2\in B(E(\theta),r_0).
\end{equation*}
More generally, for any $n \geq 1$, we call the \textit{inverse branch} of $E^n$ at $\theta$ the composition
$$E^{-n}_\theta =E^{-1}_\theta\circ E^{-1}_{E(\theta)}\circ\cdots\circ E^{-1}_{E^{n-1}(\theta)} :B(E^n(\theta),r_0)\to B(\theta,r_0)$$
of the inverse branches of $E$ at the iterates of $\theta$. Observe that
$E^{-n}_\theta(E^n (\theta)) = \theta$ and $E^n\circ E^{-n}_\theta=id$. 

In particular, for every $\theta\in \Td$, $n\geq 0$ and every $0<\epsilon\leq r_0$, we have
\begin{equation*}\label{eq:expandineq3}
E^n(B(\theta,n+1,\epsilon))=B(E^n(\theta),\epsilon).
\end{equation*}
As a consequence, any $C^1$ expanding map on a compact manifold is \textit{expansive}, with $r_0$ being the constant of expansivity: If $\dist(E^n(\theta_1),E^n(\theta_2))\leq r_0$ for every $n\geq 0$, then $\theta_1=\theta_2$.

 Recall that an application $f:M \to M$ is \textit{topologically exact} if given any nonempty open set $U \subset M$, there exist $N \geq 1$ such that $f^N(U)=M$. In particular, every expanding map is topologically exact. 

The next lemma will be useful for our purpose and can be found in \cite[Lemma 4.3]{BV2000}. In our case, we need to adapt the original statement for the Lebesgue measure to any Borel probability measure $\eta$. In the next lemma, given any disc $\gamma\subseteq M$, $N_\gamma\geq1$ denotes the minimal integer such that $E^{N_\gamma}(\gamma)=M$.

\begin{lema}\label{le:baciatotal1} Let $E:M\to M$ be an expanding map and $\eta$ any $E$-invariant Borel probability measure defined on $M$. Denote by $r_0>0$ the radius of injectivity for $E$. Let $0<r <r_0$ be fixed. Then, for any disc $\gamma\subset M$ of diameter less than $2r_0$ having a boundary of zero $\eta$-measure and for any $n\geq N_\gamma$, there exist open sets $V_i\subseteq W_i\subseteq \gamma$, $i=1,\dots j(n)$, such that
\begin{enumerate}
\item[(i)] the $ V_i$ are pairwise disjoint; 
\item[(ii)] $\eta(\cup_{i=1}^{j(n)}W_i)$ converges to $\eta(\gamma)$ as $n\to\infty$;
\item[(iii)] each $E^n(V_i)$, $i=1,\dots j(n)$, is a ball of radius $r$ and
\item[(iv)] each $E^n(W_i)$, $i=1,\dots j(n)$, is a ball of radius $2r$.
\end{enumerate}
\end{lema}
\proof We proceed as in \cite[Lemma 4.3]{BV2000}. Given any $n\geq N_\gamma$, let $B(\theta_l, r)$, $l=1,\dots, l_0$, be a maximal family of disjoint balls of radius $r>0$ inside of $M$. This means that for any $z \in M$, the ball $B(z, r)$ intersects $B(\theta_l, r)$ for some $l = 1,\dots, l_0$. In particular, the family $B(\theta_l, 2r)$ covers the whole $M$. For each $l=1,\dots, l_0$ and for each $\tilde\theta_{i,l}\in E^{-n}(\theta_l)\cap \gamma$, we denote $E^{-n}_{\tilde\theta_{i,l}}$ as the inverse branch such that $E^{n}_{\tilde\theta_{i,l}}(\tilde\theta_{i,l})=\theta_{l}$, and we define
$$V_{i,l}=E^{-n}_{\tilde\theta_{i,l}}(B(\theta_l,r))\quad \mbox{ and } \quad W_{i,l}=E^{-n}_{\tilde\theta_{i,l}}(B(\theta_l,2r)).$$
To prove item (i), we fix $l$ and take $i_1\neq i_2$, then $V_{i_1,l}\cap V_{i_2,l}= \emptyset$, because they are images of different inverse branches. Now, if we fix $i$ and take $l_1\neq l_2$, then $V_{i,l_1}\cap V_{i,l_2}= \emptyset$ because $B(\theta_{l_1},r)\cap B(\theta_{l_2},r)=\emptyset$. 
Items (iii) and (iv) are obtained by construction.

We are left to prove part (ii) of the statement. Indeed, note that the union of the $W_i$ contains the set of points whose distance to the boundary of $\gamma$ is larger than $\lambda^{-n}r$, and since the boundary of $\gamma$ has zero $\eta$-measure, then the measure of the complement of this set goes to zero as $n$ approaches infinity. 

\endproof

Let $\eta$ be a probability measure on $M$ which is not necessarily invariant under $f$. A measurable function $J_\eta f : M\to[0, \infty)$ is a \textit{Jacobian} of $f$ with respect to $\eta$ if the restriction of $J_\eta f$ to any measurable subset $A$ of an invertibility domain of $f$ is integrable with respect to $\eta$ and satisfies
$$\eta(f(A))=\int_A J_\eta f\,d\eta.$$
We say that the Jacobian of $f$ with respect to $\eta$ has \textit{bounded distortion} if there exists $C>0$ such that for every integer $n\geq 1$, every $x\in M$ and every $y\in B(x,n+1,r_0)$,
\begin{equation}\label{eq:boundeddistortion}
\frac{1}{C}\leq \frac{J_{\eta}f^{n}(x)}{J_{\eta}f^{n}(y)}\leq C.
\end{equation}
Let us note that a standard argument shows that if $f$ is an expanding map and the Jacobian is continuous H\"older, then it has bounded distortion.

Let $\eta$ be a Borel probability measure on $M$ and let $A\subseteq M$ be a measurable set. Recall that a point $z\in M$ is a $\eta$-\textit{density point} of the set $A$ if
$$\lim_{\epsilon\to 0}\frac{\eta\big(A\cap B(z,\epsilon))}{\eta(B(z,\epsilon)\big)}=1.$$
When $\eta$ is other than Lebesgue, it follows from \cite[Section 2.9]{F1969} (see also \cite[section 4]{LY1985}) 
that $\eta$-almost every point in $A$ is a $\eta$-density point. 

\begin{lema}\label{le:baciatotal2} Let $E:M\to M$ be an expanding map and let $\nu$ be any $E$-invariant Borel probability measure having a Jacobian with bounded distortion.
Let $A\subseteq \mathbb{T}^d$, $\nu (A)>0$. Then, for every $0<r<r_0$, there exists a sequence of balls of radius $r$, $\sigma_n\subseteq M$, $n\geq 1$, such that
$$\lim_{n\to \infty}\dfrac{\nu\left( \sigma_n \cap E^n(A)\right)}{\nu\left(\sigma_n\right)}=1.$$
\end{lema}

\proof Letting $\theta_0\in A$ be a $\nu$-density point of $A$, then there exists a decreasing sequence of balls $\omega_m$ around $\theta_0$ with a radius approaching 0, such that the relative measure of $A$ in $\omega_m$ approaches 1 when $m\to \infty$: that is, 
\begin{equation}\label{eq:densidad}
\lim_{m\to \infty} \dfrac{\nu(\omega_m\cap A)}{\nu(\omega_m)}=1.
\end{equation}
For each sufficiently large and fixed $m$, we choose $n=n(m)\geq N_\gamma$ such that Lemma~\ref{le:baciatotal1} can be applied to $\gamma=\omega_m$. We may assume that $n(m)\to \infty$ as $m\to\infty$. Let $V^m_{i,l}$ and $W^m_{i,l}$ and $i=1,\dots, j(m)$, $l=1,\dots,l_0$ be the open sets obtained from such a lemma. 

\noindent \textbf{Claim:} There exists a constant $0<\tau\leq 1$, independent of $m\geq 1$, such that 
\begin{equation}\label{eq:claim}
\nu\left(\bigcup_{l=1}^{l_0}\bigcup_{i=1}^{j(m)}V^m_{i,l}\right)\geq \tau\,\nu(\omega_m),\quad \mbox{ for every } m\geq 1,
\end{equation}
that is, the union of the $V^m_{i,l}$s covers a fixed fraction of $\omega_m$ for every $m$. 

Recall that $V^m_{i,l}$ and $W^m_{i,l}$, $i=1,\dots, j(m)$, $l=1,\dots,l_0$ were chosen satisfying
$$E^{n(m)}(V^m_{i,l})=B(\theta^m_l,r)\mbox{ and } E^{n(m)}(W^m_{i,l})=B(\theta^m_l,2r),$$
for certain $\theta^m_l\in \Td$, $l=1,\dots, l_0$. Equivalently, if $\tilde\theta^m_{i.l}\in V^m_{i,l}$ is such that $E^{n(m)}(\tilde\theta^m_{i,l})=\theta^m_l$, then
$$V^m_{i,l}=B(\tilde\theta^m_{i,l},n(m),r)\mbox{ and } W^m_{i,l}=B(\tilde\theta^m_{i,l},n(m),2r).$$
In particular, since $J_\nu f$ has bounded distortion, we obtain
\begin{equation*}\label{eq:proporcionV}
\frac{1}{CJ_\nu E^n(\tilde\theta^m_{i,l})} \nu\left(B(\theta^m_l,r)\right)\leq \nu\left(V^m_{i,l}\right)\leq \frac{C}{J_\nu E^n(\tilde\theta^m_{i,l})}\nu\left(B(\theta^m_l,r)\right),
\end{equation*}
analogously for $W^m_{i,l}$.
Then, 
\begin{eqnarray}
\dfrac{\nu\left(\bigcup_{l=1}^{l_0}\bigcup_{i=1}^{j(m)}V^m_{i,l}\right)}{\nu\left(\bigcup_{l=1}^{l_0}\bigcup_{i=1}^{j(m)}W^m_{i,l}\right)}
&\geq&
\dfrac{\sum_{l=1}^{l_0}\sum_{i=1}^{j(m)} \nu\left(V^m_{i,l}\right)}{ \sum_{l=1}^{l_0}\sum_{i=1}^{j(m)} \nu\left(W^m_{i,l}\right)}\\
&\geq&
\dfrac{\sum_{l=1}^{l_0}\sum_{i=1}^{j(m)} (C J_{\nu}E^n(\tilde\theta^m_{i,l}))^{-1}\,\nu\left( B(\theta^m_l,r)\right)}{\sum_{l=1}^{l_0}\sum_{i=1}^{j(m)} C(J_{\nu}E^n(\tilde\theta^m_{i,l}))^{-1}\,\nu\left( B(\theta^m_l,2r)\right)}\nonumber\\
&=&
\frac{1}{C^2}\dfrac{\sum_{l=1}^{l_0}\nu\left( B(\theta^m_l,r)\right)}{\sum_{l=1}^{l_0} \nu\left( B(\theta^m_l,2r)\right)}.\label{eq:cuociente1}
\end{eqnarray}
On the other hand, based on the fact that $\nu$ is fully supported and from the compactness of $\mathbb{T}^d$, it is straightforward to prove that for all $r> 0$, there exists $C^+_r,C^-_r>0$ such that 
\begin{equation}\label{eq:cotabola}
C^-_r\leq\nu (B (\theta, r))\leq C^+_r
\end{equation}
for all $\theta \in \mathbb{T}^d$.  
Now, including \eqref{eq:cotabola} in \eqref{eq:cuociente1}, we obtain
$$\dfrac{\sum_{l=1}^{l_0} \nu\left( B(\theta^m_l,r)\right)}{\sum_{l=1}^{l_0} \nu\left( B(\theta^m_l,2r)\right)}
\geq\frac{C^-_r}{C^2\,C^+_{2r}}.
$$
Taking $\tau=C^-_r/C^2\, C^+_{2r}$, the claim is proven.

Finally, Lemma~\ref{le:baciatotal1} part (ii) completes the proof of the claim. Indeed, since the $V^m_{i,l}$ are pairwise disjoint, it follows from the \eqref{eq:claim} above combined with \eqref{eq:densidad} that we can choose some $i_m\in\{1,\dots,j(m)\}$ and $l_m\in{1,\dots,l_0}$ such that the sequence $V^m_{i_m,l_m}$ satisfies
\begin{equation}\label{eq:bolas}
\lim_{m\to \infty} \dfrac{\nu(V^m_{i_m,l_m}\cap A)}{\nu(V^m_{i_m,l_m})}=1.
\end{equation}
We take $\sigma_{n(m)}=E^{n(m)}(V^m_{i_m,l_m})$, and again using \eqref{eq:boundeddistortion}, the injectivity of $E^{n(m)}$ in $V^m_{i_m,l_m}$ and \eqref{eq:bolas}, we have
$$\dfrac{\nu\big(\sigma_{n(m)}\cap E^{n(m)}(A)\big)}{\nu\big(\sigma_{n(m)}\big)}
=
\dfrac{\nu\big(E^{n(m)}(V^m_{i_m,l_m}\cap A)\big)}{\nu\big(E^{n(m)}(V^m_{i_m,l_m})\big)} 
=\dfrac{ k^{n(m)} \nu\big(V^m_{i_m,l_m}\cap A)\big)}{k^{n(m)} \nu \big(V^m_{i_m,l_m}\big)}\to 1,$$
and therefore, the proof concludes. \endproof

As an immediate consequence of Shub's theorem, $\htop(E)=\log k$ and there exists a unique measure of maximal entropy $\nu_0$ for $E$: the push forward of the Lebesgue measure by the topological conjugacy. Nevertheless, equilibrium states for a large class of potentials can be obtained from the following classical theorem due to David Ruelle (see \cite[Theorem 12.1]{OV2016}). 

\begin{teo}[Ruelle]\label{teo:Ruelle} Let $E : M\to M$ be an expanding map and $\Phi : M \to\mathbb R$ be a H\"older function. Then, there exists a unique equilibrium state $\nu_\Phi$ for $\Phi$. Moreover, the measure $\nu_\Phi$ is exact and is supported on the entire $M$ and $J_{\nu_\Phi} E$ is H\"older continuous.
\end{teo}

In the particular case in which $f$ is $C^{1+\alpha}$, $\alpha>0$, the equilibrium state of the potential $\Phi = -\log |\det Df|$ coincides with the absolutely continuous invariant measure. In particular, it is the unique physical measure of $f$. The measure of maximal entropy is obtained when considering the potential $\Phi=0$ in Theorem~\ref{teo:Ruelle}.

The following result is deduced from the proof of Theorem~\ref{teo:Ruelle} and is useful for our purpose because it provides us with a large set of $E$-invariant measures with bounded Jacobian.

\begin{lema}[Corollary 12.1.17 \cite{OV2016}] Let $E : M\to M$ be anexpanding map and $\Phi : M \to\mathbb R$ be a H\"older function. There exists $C>0$ such that for every equilibrium state $\nu_{\Phi}$ for $\Phi$, every $n\geq 1$, every $x\in M$ and every $y\in B(x,n+1,r_0)$,
$$\frac{1}{C}\leq \frac{J_{\nu_{\Phi}}E^{n}(x)}{J_{\nu_{\Phi}}E^{n}(y)}\leq C.$$
\end{lema}

We recall that for any expanding map $E$ defined on the connected manifold $M$ (as in our case), the periodic points are dense in $M$ (see \cite[Corollary 11.2.16.]{OV2016}, and every invariant measure can be approximated by measures supported on periodic points. Furthermore, the topological entropy coincides with the growth rate of the number of periodic points. See \cite[Proposition 11.3.2,Theorem 11.3.4]{OV2016}.

\begin{lema}\label{le:periodicpoints1}
Let $E : M \to M$ be an expanding map. Then,
$$\lim_{n\to \infty} \frac1n \log \# \mathrm{Fix} (E^n) = \htop (E).$$
Moreover, every probability measure $\mu$ invariant under $E$ can be approximated in the weak* topology by invariant probability measures supported on periodic orbits.
\end{lema}

At the end, we recall that in the case of the  uniformly expanding  maps, we have a certain \textit{statistical stability} of the  measure maximizing the entropy.

\begin{lema}\label{le:statest}
Let $E\::\Td\to\Td$ be a $C^r$ expanding map, $r\geq 1$. Let $\nu_0$ be the unique measure maximizing the entropy of $E$. Let $E_n:\Td\to\Td$ be a sequence of expanding maps, $E_n\to E$, as $n\to \infty$ in the $C^2$-topology. Let us denote by $\nu_n$ the sequence of measures maximizing the entropy of $E_n$, $n\geq1$. Then, $\nu_n\to \nu_0$ as $n\to\infty$ in the weak* topology. 
\end{lema}

 In \cite{VV2010}  we can find a more general statement of statistical stability for non-uniformly expanding local diffeomorphism and equilibrium states. Specifically, the generalization of statistical stability for equilibrium states is shown.  Let $\mathcal{F}$ be a family of local homeomorphisms with Lipschitz inverse and $\mathcal W$  be some family of continuous potentials $\phi$. A pair $(f, \phi)\in \mathcal{F} \times \mathcal{W}$ is \emph{statistically stable} (relative to $\mathcal{F} \times \mathcal{W}$) if, for any sequence $f_n \in\mathcal{F}$ converging to $f$ in the uniform topology, with its Lipschitz constants, $L_n$ converging to a $L$ in the uniform topology and $\phi_n\in \mathcal{W}$ converging to $\phi$ in the uniform topology, and for any choice of an equilibrium state $\mu_n$ of $f_n$ for $\phi_n$, every weak* accumulation point of the sequence $(\mu_n)_{n\geq1}$ is an equilibrium state of $f$ for $\phi$. In particular, when the equilibrium state is unique, statistical stability means that it depends continuously on the data $(f, \phi)$.
 
 In our setting, we can write it as follows.

\begin{teo}[Theorem D. \cite{VV2010}]\label{teo:StabstatVV}
Let $E\: : \:\Td\to\Td$ be an expanding map  and $\phi : M\to \mathbb{R}$ is a H\"older continuous potential such that 
\begin{equation}\label{eq:boundvar}
\sup \phi - \inf \phi < \htop (f).
\end{equation}

Then the pair $(E,\phi)$ is statistically stable.
\end{teo}


Note that if we fix $E$ and $\phi\equiv 0$ the null potential, the hypotheses of the theorem are satisfied 
and we recover the Lemma~\ref{le:statest}.



\subsection{Partial hyperbolic endomorphisms}\label{ssec:hiperbparcial}

Let $f:M\to M$ be a local diffeomorphism on $M$. Consider a splitting of the tangent bundle $TM= E\oplus F$. This splitting is not assumed to be  invariant by $Df$.
The \textit{cone field of width} $a>0$ \textit{around} $E$ is the family
$C_a(E)=\{C_a(E,x)\::\: x\in E\}$, where
$$C_a(E,x)=\{v^E+v^F\in E\oplus F\::\: \frac{\|v^F\|}{\|v^E\|}\leq a\}.$$
The family $C_a(E)$ is an \textit{unstable cone field} if there exists $\lambda>1$ such that

\begin{enumerate}
    \item[{[CF1]}] $C_a(E)$ is $Df$-forward invariant: For every $x\in M$,
    $$ Df(x)C_a(E,x)\subseteq C_{ a/\lambda}(E,f(x)).$$
    \item[{[CF2]}] $Df$ uniformly expands the vector inside the cone: For every $x\in M$, $v\in C_a(E)$, and $n\geq 1$, 
    $$\|Df^n(x)v\|\geq \lambda^n \|v\|.$$
\end{enumerate}

Note that $F$ can be chosen dynamically as defined by
$$F(x)=\bigcap_{n\geq 1}T_xM\setminus [(Df^n(x))^{-1}C_a(E,f^n(x))],$$
resulting in a $Df$-invariant subbundle. Furthermore, we have as consequence of the definition that the rate of growth of vectors in $F$ by the derivative is dominated by the rate of growth of vectors in $E$. Hence, up to a change of the metric, we can assume that for every $x\in M$ and $v^F\in F$ 
$$\|Df(x)v^F\|\leq\frac{\lambda}2\|v^F\|.$$
A local diffeomorphism $f\::\: M\to M$ is  \textit{partially hyperbolic} if there exists a Riemannian metric on $M$, there are constants $a>0$ and $\lambda>1$ and there is a splitting $TM=E^u\oplus E^c$ such that $C_a(E^u)$ is an unstable cone field of width $a>0$ around $E^u$ with rate of expansion $\lambda>1$ and $E^c$ is the $Df$-invariant subbundle defined as above.
    
 It follows from the definition that partial hyperbolicity is a $C^r$-open property, $r\geq 1$. 
 
We now precisely state condition \ref{F2}: 
 
 \begin{itemize}
 \item[\referencia{F2}{[F2]}] There exists a Riemannian metric on $M$, there are constants $a>0$ and $\lambda>1$ and there is a splitting $TM=E^u\oplus E^c$, where $E^u=\mathbb{R}^d$ such that $C_a(E^u)$ is an unstable cone field of width $a>0$ around $E^u$ with rate of expansion $\lambda>1$ and $E^c$ is the $DF$-invariant subbundle defined as above.
 \end{itemize}

\subsection{Normally hyperbolic center foliation}\label{ssec:foliaciones}
A \textit{foliation} $\mathcal F$ of a manifold $M$ is a division of $M$ into disjoint submanifolds called leaves of the foliation with the following properties.
\begin{itemize}
\item Each leaf is connected, although it need not be a closed subset of the manifold.
\item The leaves all have the same dimension.
\item For each point $p\in M$ there exists a neighborhood $V_p$ of $p$, a homeomorphism $\xi:D^c \times D^{m-c}\to V_p$ such that
$$\phi(D^c\times\{y\})\subseteq \mathcal F(\xi(0, y)).$$
\end{itemize}
$D^c$ is the open $c$-dimensional disc, $m$ is the dimension of the manifold $M$, and $c$ is the dimension of the foliation. Such a $\phi$ is called a \textit{foliation box} and $\phi(D^c\times\{ y\})$ is called a \textit{plaque} of the foliation. A foliation $\mathcal F$ on $M$ is called \textit{regular} if $\partial_1\phi$ exists, is nonsingular, and depends continuously on $p$. For a regular foliation $\mathcal F$ of dimension $c$, consider finitely many foliation boxes $\{\phi\::\:D^c\times D^{m-c}\to M\}$ such that $\phi(\frac12 D^c\times \frac12D^{m-c})$ covers $M$. The family $\mathcal P=\{\phi(D^c\times \{y\})\}$ of plaques is then called a \textit{plaquation} of $\mathcal F$.

Let $f:M\to M$ be a partially hyperbolic local diffeomorphism on $M$. The foliation $\mathcal F$ is $f$-\textit{invariant} if $f$ permutes its leaves. That is,
$$f\circ\pi=\pi\circ f$$
where $\pi$ projects the point $p\in M$ to the leaf $\mathcal F(p)$ containing it.

It is known that in general, the center distribution $E^c$ does not necessarily integrate to a foliation. Even if $E^c$ is integrable, in general, the foliation $\mathcal F^c$ does not share regularity properties such as absolute continuity or H\"older continuity. The natural question of the existence and  regularity of the center foliation is tackled in the classical monograph of M. W. Hirsch, C. C. Pugh and M. Shub \cite{HPS1977}. For a review of the recent advances in answering this question we refer the reader to \cite{PSW2012}.

A central result in \cite{HPS1977} establishes that $r$-normal hyperbolicity and plaque expansiveness are sufficient conditions to obtain existence and some regularity of center foliations in the precise sense described below.

A regular $f$-invariant foliation $\mathcal F$ with $C^r$ leaves, is called \textit{$r$-normally hyperbolic} with $r\geq 1$, if there is a $Df$-invariant splitting $TM=E^u\oplus E^c\oplus E^s$, $E^c=T\mathcal F$, such that for any $p\in M$,
$$ \Vert D f_p|E^s\Vert < 1<\Vert (Df_p|E^u)^{-1}\Vert^{-1},$$
and 
$$ \Vert D f_p|E^s\Vert < \Vert Df_p|E^c\Vert ^r<\Vert (D f_p|E^u)^{-1}\Vert^{-1}. $$

 A $C^r$-partially hyperbolic map $f$ is called \textit{$r$-partially hyperbolic} if the central bundle $E^c$ integrates into a foliation $\mathcal F^c$ which is $r$-normally hyperbolic.

An $f$-invariant regular foliation $\mathcal F$ is called \textit{plaque expansive} if there exist a plaquation $\mathcal P$ and a $\delta>0$ such that any two $\delta$-pseudo orbits of $f$ which respect the plaquation $\mathcal P$ and $\delta$-shadow each other belong to the same plaque of $\mathcal P$, where the $\delta$-pseudo orbit $\{x_i\}$ is said to respect $\mathcal P$ if $f(x_i)$ and $x_{i+1}$ belong to the same plaque of $\mathcal P$. An $f$-invariant normal hyperbolic foliation $\mathcal F$ exhibits plaque expansivity if the foliation $\mathcal F$ is of class $C^1$ \cite{HPS1977}, or if the leaves of $\mathcal F$ have uniformly bounded leaf volume \cite{C2011}. Nevertheless, the fundamental question of whether normal hyperbolicity implies plaque expansivity in general remains open.

A \textit{leaf conjugacy} from an $f$-invariant foliation $\mathcal F_f$ to a $g$-invariant foliation $\mathcal F_g$ is a homeomorphism $h : M \to M$ sending $\mathcal F$-leaves to $\mathcal G$-leaves and
$$h(f(\mathcal F_f(p)))=g(\mathcal F_g(h(p))).$$
The result in \cite{HPS1977} about foliation stability establishes that if $\mathcal F_f$ is $r$-normally hyperbolic and plaque expansive with respect to an $f$, then for each $g$ $C^r$-close to $f$, there exists a unique $g$-invariant foliation $\mathcal F_g$ near $\mathcal F_f$. The foliation $\mathcal F_g$ is normally hyperbolic, plaque expansive with respect to $g$ and $\mathcal F_f$ is conjugate to $\mathcal F_g$ by a homeomorphism $h$ such that $h(f(x))$ and $g(h(x))$ belong to the same leaf of $\mathcal F_g$.
Moreover, every leaf of $\mathcal F_g$ is $C^r$ smooth and $C^r$-close to the leaf of $\mathcal F_f$.

We can now precisely state hypothesis \ref{F3}. Recall that we consider $F: \Td\times I\to \Td\times I$, being a $C^r$-partial hyperbolic local diffeomorphism preserving the boundary, $r\geq 1$, defined by
$$ F(\theta,t)=(E(\theta,t),\varphi(\theta, t)),$$
We assume that $F$ satisfies: 
\begin{itemize}
\item[\referencia{F3}{[F3]}] There exists a unique $F$-invariant foliation $\mathcal F^c_F$ tangent to $E^c$ such that
\begin{itemize}
\item[\referencia{F3i}{[F3.i]}] For every $p\in M$, the leaves $\mathcal F^c_F(p)$ are curves of uniformly bounded length $C^1$-close of $\{\theta\}\times I$,
\item[\referencia{F3ii}{[F3.ii]}] the map $p\to \mathcal F^c(p)$ is H\"older continuous.
\end{itemize}
\end{itemize}

Note that \ref{F3i} implies that $\mathcal F^c$ is normally hyperbolic and plaque expansive. If $K$ is a Kan-like map defined by \eqref{eq:Kanmap}, then $K$ satisfies \ref{F3}. In fact, the one-dimensional subbundle $E^c=\{0\}\times\mathbb R$ integrates uniquely in the $K$-invariant foliation $\mathcal{F}^c_K=\{I_\theta=\{\theta\}\times I\::\:\theta\in\Td\}$. It follows from \ref{K2} that $\mathcal F^c_K$ is normally hyperbolic and plaque expansive. We can therefore formulate the following perturbation result which is useful for our purpose.



\begin{teo}[\cite{HPS1977,IN2012,PSW2012}]\label{teo:foliaciones}
There exists a $C^1$-open set $\mathcal U\subseteq {\rm End}^r(M,\partial M)$ containing $K$ such that every $F\in \mathcal U$ satisfies:
\begin{enumerate}
\item There exists a unique $F$-invariant foliation $\mathcal F^c_F$ near $\mathcal F^c_K$ satisfying \ref{F3}. 
\item There exists a H\"older continuous leaf conjugacy between $(K,\mathcal F_K^c)$, $(F,\mathcal F_F^c)$.
\item The holonomy map along the leaves of $\mathcal F_F^c$ is H\"older continuous.
\end{enumerate}
Moreover,
\begin{enumerate}
\item[(4)] if $K$ is $r$-partially hyperbolic, then the leaves of $\mathcal F^c_F$ are $C^r$ and they are $C^r$-close to the leaves of $\mathcal F^c_K$.
\end{enumerate}
\end{teo}


\section{Measures supported on the boundaries}\label{sec:mmei} 
This section is devoted to proving  Theorem~\ref{mainteo:KE1}, part~\ref{TeoAi}.  We include it as an independent proposition for future reference.

We assume throughout the whole section that $K$ is a Kan-like map. Let $\nu$ be an invariant measure for the expanding map $E$ and fix $m=\nu\times \m_I$ as a reference measure on the cylinder $M=\Td\times I$. We denote by $\pi:M\to \Td$ the natural projection defined by $\pi(\theta, t)=\theta$, $\theta\in\Td$.

\begin{prop}\label{teo:medidasnegativas}
Assume $K$ is a Kan-like map and suppose that $\nu$ is any ergodic $E$-invariant probability measure on $\Td$ satisfying condition \eqref{eq:condK4im}: 
$$\int\log |\partial_t \varphi (\theta,j)|\: d\nu(\theta)<0, \quad j=0,1.$$
The measures $\mu_j=\nu \times \delta_j$ are then $K$-invariant and ergodic, have negative central Lyapunov exponents, are supported on $\Td \times \{j\}$, and satisfy the following properties:
    \begin{itemize}
        \item[(a)] The measures $\mu_j$ are observable measures with respect to $m$.
    
        \item[(b)] If $\supp \nu =\Td$, then the basins $\mathcal{B}(\mu_j)$ are intermingled with respect to $m$.

        \item[(c)] If $E$ has a Jacobian with respect to $\nu$ with bounded distortion, then the union of the basins $\mathcal{B}(\mu_0)\cup \mathcal{B}(\mu_1)$ covers $m$-a.e. the manifold $M$.
        
        \item[(d)] There is no other invariant measure $\mu$ such that $\pi_*\mu=\nu$ with a negative center Lyapunov exponent. 
    \end{itemize}

\end{prop}

The proof of the Proposition above hinges on a series of partial results proved below. We summarize the proof in the Subsection~\ref{ssec:mmei}.

\subsection{$\mu_0$ and $\mu_1$ have negative Lyapunov exponents}\label{ssec:Lyapunov2} 

Let $\eta$ be an ergodic $K$-invariant probability measure. The \textit{central Lyapunov exponent} of $K$ with respect to $\eta$ is
\begin{equation}\label{eq:deflyapucentral}
\lambda^c(K,\eta):=\int_M\log\|DK(\theta,t)(0,1)\|\,d\eta(\theta,t).
\end{equation}
For simplicity, from now on we write $\lambda^c(\eta)$ instead of $\lambda^c(K,\eta)$. 

\begin{lema}\label{le:expnentenegativo} Let $\nu$ be an ergodic $E$-invariant probability measure such that
$$ \int_{\Td}\log|\partial_t \varphi (\theta,j)|\,d\nu(\theta)<0, \quad j=0,1.$$
Then, the probability $\mu_j = \nu \times \delta_j $ is $ K $-invariant, ergodic and has a center Lyapunov exponent 
$$ \lambda^ c(\mu_j) <0.$$
\end{lema}
\proof It is immediately clear that $\mu_j=\nu\times \delta_j$ is $ K $-invariant, ergodic and supported on the boundary $\Td\times\{j\}$. From \eqref{eq:deflyapucentral}, we have
\begin{eqnarray*}
\lambda^ c(\mu_j)&=&\int_M\log\|DK(\theta,t)(0,1)\|\,d (\nu\times \delta_j)(\theta,t)\\
&=&\int_{\Td}\log\|DK(\theta,j)(0,1)\|\,d\nu(\theta).
\end{eqnarray*}
We notice that 
$$DK(\theta,j)=\left(\begin{array}{ll}
\partial_\theta E(\theta)& 0\\
\partial_\theta \varphi (\theta,j) & \partial_t \varphi (\theta,j)
\end{array}
\right)$$
\noindent the central space $E^c(\theta,j)=\{0\}\times\mathbb{R}$ is invariant, and we also have 
$$\log\Vert DK(\theta,j)(0,1)\Vert=\log|\partial_t \varphi (\theta,j)|.$$
and then,
$$ \lambda^c(\mu_j)=\int_{\Td}\log|\partial_t \varphi (\theta,j)|\,d\nu(\theta)<0.$$
\endproof

It follows from Lemma~\ref{le:expnentenegativo} that $\nu$-almost every point in the boundary has a Pesin stable manifold.

\begin{lema}\label{lem:intervalos} Let $\nu$ be an ergodic $E$-invariant probability measure such that
$$\int_{\Td}\log|\partial_t \varphi (\theta,j)|\,d\nu(\theta)<0.$$
Then, for $\nu$-almost every $\theta\in\Td$, the Pesin stable manifold in $(\theta,j)$ is an interval inside $I_\theta=\{\theta\}\times I$. 
\end{lema}

\proof Letting $\theta \in \Td$, notice that the map $\varphi_\theta : I_\theta \to I_{E(\theta)}$ satisfies \ref{K1} and is a diffeomorphism. Then, $\varphi_\theta$ is strictly increasing, as are all of its iterated $\varphi_n:I_\theta \to I_{E^n(\theta)}$ defined by 
$$\varphi_n(t):=K^n(\theta,t).$$ 
Assume that the stable manifold of $(\theta,0)$ contains the maximal interval $\{\theta\}\times[0,\sigma_0(\theta))$, $\sigma_0(\theta)\geq0$.
We will prove that $W^s((\theta,0))=\{\theta\}\times[0,\sigma_0(\theta))$, $\nu$-a.e. $\theta$. In fact, let $t_0>\sigma_0(\theta)$. 
Since the map $\varphi_n$ is increasing, we have for every $\sigma_0(\theta)<t<t_0$ 
$$\varphi_n(\theta,\xi_0(\theta)) < \varphi_n(\theta,t)< \varphi_n(\theta,t_0). $$
Therefore, if we assume that $(\theta,t_0)\in W^s((\theta,0))$, we contradict the fact that $\sigma_0(\theta) $ is the end of the maximal interval, and the result is obtained. Similarly, it is shown that $W^s((\theta,1))=\{\theta\}\times(\sigma_1(\theta),1]$, $\nu$-a.e. $\theta$.

\endproof


\subsection{$\mu_0$ and $\mu_1$ are observable measures }\label{ssec:obsmeasures2}
In this subsection, we assume that $\nu$ is an $E$-invariant probability measure on $\Td$ such that~\eqref{eq:condK4im} is satisfied. Now, we prove that $\mu_j$, $j=0,1$ is an observable measure with respect to the reference measure $m=\nu\times \m_I$.

\begin{lema}\label{le:existenfisicas} The measure $\mu_j=\nu\times\delta_j$, $j=0,1$, is observable with respect to the reference measure $m=\nu\times \m_I$. 
\end{lema}

\proof We notice that $\Td\times\{j\}$ is a submanifold transversal to $E^c_{(\theta,0)}$ for every $\theta\in\Td$. Since $\nu$ is assumed to be ergodic and supported on $\Td$, then the ergodic measure $\mu_j$ has support on $\Td\times\{j\}$. Thus, there exists a $\nu$-full measure subset of $\theta\in\Td$ such that 
$(\theta,j)\in \mathcal{B}(\mu_j)$.

On the other hand, from \eqref{eq:condK4im}, we have the conclusions of Lemma~\ref{le:expnentenegativo} and Lemma~\ref{lem:intervalos}, and so $\mu_j$ has a negative center Lyapunov exponent and the Pesin stable manifold in $(\theta,j)$ is an interval inside $I_\theta=\{\theta\}\times I$ for $\nu$-almost every $\theta\in\Td$. Therefore, we conclude that
$$\mu_j\bigg(\big(\Td\times\{j\}\big)\cap\mathcal{B}(\mu_j)\cap\mathcal{R}(\mu_j)\bigg)=1, \quad j=0,1.$$
In particular, for $\epsilon>0$ which is fixed and small enough, there exists a set $\Gamma_j\subseteq \Td$ such that $\nu(\Gamma_j)>0$ and we have 

\begin{itemize}
\item  $\{\theta\}\times [0,\epsilon]\subseteq W^s(\theta,0)$ for every $\theta\in \Gamma_0$, and 
\item  $\{\theta\}\times [1-\epsilon,1]\subseteq W^s(\theta,1)$ for every $\theta\in \Gamma_1$. 
\end{itemize}

Moreover, 
$$\Gamma_j\times\{j\}\subseteq\mathcal{B}(\mu_j)\cap\mathcal{R}(\mu_j).$$ 
An interesting property of Pesin's stable manifolds is that if $x\in\mathcal{B}(\mu)$, then $W^s(x)\subseteq\mathcal{B}(\mu)$. In other words, $\mathcal{B}(\mu)$ is an \textit{$s$-saturated} set. 

Therefore, 
$$\Gamma_0\times [0,\epsilon] \subseteq \mathcal B(\mu_0)\quad\mbox{ and }\quad\Gamma_1\times[1-\epsilon,1] \subseteq \mathcal B(\mu_1).$$
\noindent Hence,
$$m(\mathcal B(\mu_j))\geq \nu(\Gamma_j)\cdot\epsilon>0,$$
\noindent where we conclude that $\mu_j$ is an observable measure with respect to $m$. \endproof

For future reference, we denote by:
\begin{equation}\label{eq:caja1bacia}
B_0(\epsilon):=\Gamma_0\times [0,\epsilon]\subseteq \mathcal B(\mu_0) \quad\mbox{ and }\quad m(B_0(\epsilon))=\nu_0(\Gamma_0)\cdot\epsilon>0.
\end{equation}
and
\begin{equation}\label{eq:caja2bacia}
B_1(\epsilon):=\Gamma_1\times [1-\epsilon,1]\subseteq \mathcal B(\mu_1) \quad\mbox{ and }\quad m(B_1(\epsilon))=\nu_0(\Gamma_1)\cdot\epsilon>0.
\end{equation}


\subsection{$\mu_0$ and $\mu_1$ are intermingled}\label{ssec:intermingledmeasures2}

 In this section, we prove that $\mu_0$ and $\mu_1$ are intermingled with respect to $m=\nu\times\m_I$: that is,
for every open set $U\subseteq M$, we have
$$ m(\mathcal{B}(\mu_0)\cap U)>0\quad\mbox{ and } \quad m(\mathcal{B}(\mu_1)\cap U)>0.$$
It follows from the definition above that it is necessary to require $\nu$ to be fully supported on $\Td$.

\begin{lema}\label{le:baciasintermingled} The basins of $\mu_0$ and $\mu_1$ are intermingled with respect to $m=\nu\times\m_I$.
\end{lema}
 
\proof Take any disc $\gamma \subset \Td$  and $t \in (0,1)$. We denote $\gamma_t = \gamma \times \{ t\}$, the disc inside the open cylinder $\Td\times (0,1)$ transverse to the $t$-direction. The dynamics along the $\theta$-direction are given by the map $E$, so it is uniformly expanding. From \ref{K3}, we conclude that the map $K$ is partially hyperbolic: the derivative in the $t$-direction is dominated by the derivative in the $\theta$-direction. Then, the size of $\gamma_t$ grows exponentially rapidly and the angle between (the tangent space of) $\gamma_t$ and the $t$-direction is bounded far away from $0$ as the disc $\gamma_t$ is successively iterated. Then, up to a forward iteration $n\geq 1$ and denoting by $\gamma_t^n:=K^n(\gamma_t)$, we have that $\gamma_t^n$ crosses (transversely) the segments $W^s(p,0)=\{p\}\times [0,1)$ and $W^s(q,1)=\{q\}\times (0,1]$, and so the orbit of $\gamma_t$ accumulates both boundary components of the cylinder $\Td \times [0,1]$.

Given $\epsilon>0$, taking a sufficiently large forward iteration, we can assume that $\gamma_t^n$ transversely intersects the rectangles $B_j(\epsilon)$, $j=0,1$, defined by \eqref{eq:caja1bacia} and 
$$\nu\big(\pi(\gamma_t^n\cap B_j(\epsilon))\big)>0.$$ 
Note that $\gamma_t^n\cap B_j(\epsilon)$ and $\pi(\gamma_t^n\cap B_j(\epsilon))$ are subsets of $\mathcal B(\mu_j)$, and since the basin is an invariant set, we have 
$$K^{-n}\big(\gamma_t^n\cap B_j(\epsilon)\big)\cap \gamma_t \subseteq \mathcal B(\mu_j).$$
Moreover, since $\gamma_t$ is contained in a horizontal section in the product space, we have
$$\nu\Bigg(\pi \bigg(K^{-n}\big(\gamma_t^n\cap B_j(\epsilon)\big)\cap \gamma_t \bigg) \Bigg)=\nu\Bigg(E^{-n}\bigg(\pi(\gamma_t^n\cap B_j(\epsilon) \big)\bigg)\cap\gamma \Bigg)>0.$$
The inequality is obtained because $E^n$ is a local diffeomorphism and $\nu$ is fully supported. Indeed, $\pi(\gamma^n_t)$ has a non-empty interior, and since $E^n$ is a local diffeomorphism then $E^{-n}(\pi(\gamma^n_t))$ also has a non-empty interior. Since $\supp\nu=\Td$, we conclude that $\nu(E^{-n}(\pi(\gamma^n_t))\cap \gamma)>0$. 


Because the reference measure $m$ is a product measure and $t\in (0,1)$ was chosen arbitrarily, Fubini's theorem completes the argument.

\endproof


\subsection{The basins cover almost all}\label{ssec:basinsmeasures} 
In this subsection, we prove that the basins cover $m$-almost every point of the whole manifold, where $m=\nu \times \m_I$ is our reference measure. Precisely, for the proof of Lemma~\ref{le:baciatotal2} below, we require addition of the assumption that $K$ has a Jacobian with respect to $\nu$ with bounded distortion. Our proof is inspired by the proof of the same fact for physical measures of partially hyperbolic diffeomorphisms whose center direction is mostly contracting in \cite{BV2000}.

In the proof of Lemma~\ref{le:baciasintermingled} we prove that, given any disc $\gamma$ in $\Td$ and $t\in(0,1)$, the set $\gamma_t=\gamma \times \{t\}$ intersects each basin $\mathcal B(\mu_j)$, $j=0,1$, in a set $\gamma_t^{j}$ such that its projection $\pi(\gamma_t^j)$ has a positive $\nu$-measure. We observe this fact in a more general situation, as the following lemma asserts.  

\begin{lema}\label{le:curvasc1} Let $\sigma \subseteq M$ be the graph of a $C^1$ function $G:\gamma\to[0,1]$ defined on the closed disc $\gamma \subseteq\mathbb T^d$. Then, 

\begin{itemize}
\item[(i)] there exists $0<\delta\leq1$ such that
$$ \nu \bigg(\pi\big(\sigma\cap[\mathcal B(\mu_0)\cup\mathcal B(\mu_1)]\big)\bigg)\geq \delta. $$
Moreover,
\item[(ii)] there exists $\epsilon>0$ such that for every $\tilde G:\gamma \to[0,1]$ $\epsilon$ close to $G$ in the $C^1$ topology, its graph $\tilde\sigma$ satisfies
$$ \nu \bigg(\pi\big(\tilde\sigma\cap[\mathcal B(\mu_0)\cup\mathcal B(\mu_1)]\big)\bigg)\geq \delta/2. $$
\end{itemize}

\end{lema}

\proof Note that if $\sigma\subseteq \mathbb T^d\times \{1\}$, there is nothing to prove, since 
$$\nu\bigg(\pi\big(\sigma\cap[\mathcal B(\mu_0)\cup\mathcal B(\mu_1)]\big)\bigg)=\nu\bigg(\pi\big(\sigma \cap\mathcal B(\mu_1)\big)\bigg)=\nu\big(\pi(\sigma)\big)>0$$
and then it is sufficient to set $\delta=\nu\big(\pi(\sigma)\big)$.

Now, assume that there exists $t^*<1$ such that $G(\theta)<t^*$ for every $\theta\in \gamma$. We denote $\gamma_{t^*}=\gamma \times\{t^*\}$.  From the proof of Lemma~\ref{le:baciasintermingled}, we conclude that the set $\gamma_{t^*}^0=\gamma_{t^*}\cap\mathcal B(\mu_0)$ satisfies
$$\delta:=\nu\bigg(\pi\big(\gamma_{t^*}\cap\mathcal B(\mu_0)\big)\bigg)=\nu(\pi(\gamma_{t^*}^0))>0$$
Since $\pi(\gamma_{t^*}^0)$ is formed by $\theta\in\gamma$ such that $W^s(\theta,0)\supseteq \{\theta\}\times [0,t^*]$ and $G(\theta)<t^*$ for every $\theta \in\gamma$, then 
$$\pi(\gamma_{t^*}^0)\subseteq \pi\big(\sigma\cap\mathcal B(\mu_0)\big)$$
from which we obtain 
\begin{equation*}\label{eq:intersecbacia2}
\nu \bigg(\pi\big(\sigma\cap\mathcal B(\mu_0)\big)\bigg)\geq\nu(\pi(\gamma_{t^*}^0))=\delta>0.
\end{equation*}
Of course, it follows from the previous argument that the lower bound $\delta>0$ is the same for any graph of a $C^1$ function defined on $\sigma$ and bounded by $t^*$.

In the general case, we have that $\sigma$ is not completely contained in $ \mathbb T^d\times \{1\}$, but perhaps $\sigma \cap ( \mathbb T^d\times \{1\})\ne\emptyset$, i.e., $G$ is not a constant equal to 1. Then, by continuity, there exists a closed disc $\gamma^-\subseteq \gamma$ such that $ G(\theta)<t^*<1$ for every $\theta\in\gamma^-$. Denoting by $\sigma^-$ the graph of $G$ restricted to $\gamma^-$ and $\gamma^-_{t^*}=\gamma^- \times \{t^*\}$, we have $\pi(\sigma^-\cap\mathcal B(\mu_0))\supseteq \pi(\gamma^-_{t^*}\cap\mathcal B(\mu_0))$ and, as previously proven, we obtain that 
$$\nu\bigg(\pi\big(\sigma^-\cap\mathcal B(\mu_0)\big)\bigg)\geq\nu\bigg(\pi\big(\gamma^-_{t^*}\cap\mathcal B(\mu_0)\big)\bigg)>0.$$
It is sufficient to set $\delta=\nu\big(\pi\big(\gamma_{t^*}^{-}\cap\mathcal B(\mu_0)\big)\big)>0$ to conclude the proof of (i).

To prove (ii), we have the same cases. If $\sigma\subseteq \mathbb T^d\times \{1\}$, take $\epsilon>0$ small enough such that the set $\Gamma_1$ of points $\theta\in\gamma$, such that $(\theta,1)$ has a local stable manifold of size greater than $\epsilon$ and $\nu(\Gamma_1)>\delta/2$. If $\tilde\sigma$ is the graph of a $C^1$-function close to $\sigma$, then 
$$\pi(\tilde\sigma\cap B_1(\epsilon))=\pi(\sigma\cap B_1(\epsilon))$$
where $B_1(\epsilon)=\Gamma_1\times[1-\epsilon,1]$. Then,
$$\nu(\pi(\tilde\sigma\cap B_1(\epsilon))=\nu(\pi(\sigma\cap B_1(\epsilon)))=\nu(\Gamma_1)>\delta/2.$$
In the case that there exists $t^*<1$ such that $G(\theta)<t^*$ for every $\theta\in \gamma$, it is sufficient to take
$$\epsilon=\inf\{ t^*-G(\theta)\::\: \theta\in\gamma\}>0. $$
Then, for any $\tilde G:\gamma\subseteq\mathbb T^d\to[0,1]$ $\epsilon$-close from $G$ in the $C^1$-topology, its graph $\tilde\sigma$ satisfies
$\tilde G(\theta)<t^*$ for every $\theta\in \gamma$. Arguing as above, since the graph $\tilde\sigma$ is below the graph $\gamma_{t^*}=\gamma\times\{t^*\}$, we have
$$\nu\bigg(\pi\big(\tilde\sigma\cap\mathcal B(\mu_0)\big)\bigg)=\delta>0.$$
Finally, for the remaining case, we choose $\epsilon=\inf\{  t^*-G(\theta) \::\: \theta\in\gamma^-\}>0$, where $\gamma^-\subseteq \gamma$ is a closed disc such that $ G(\theta)<t^*<1$ for every $\theta\in\gamma^-$ defined before. Then, for any $\tilde G:\gamma\subseteq\mathbb T^d\to[0,1]$ $\epsilon$-close from $G$ in the $C^1$-topology, its graph $\tilde\sigma$ is below the graph $\gamma_{t^*}=\gamma\times\{t^*\}$ when $\tilde\sigma$ is restricted to $\gamma^-$. Then,
$$\nu\bigg(\pi\big(\tilde\sigma\cap\mathcal B (\mu_0)\big)\bigg)\geq \nu\bigg(\pi\big(\tilde\sigma\cap\mathcal B (\mu_0)\big)\cap \gamma_{t^*}\bigg)=\nu\bigg(\pi\big(\gamma_{t^*}^-\cap\mathcal B (\mu_0)\big)\bigg)=\delta>0$$
\endproof

\begin{lema}\label{le:baciasestodo} The set $\mathcal{B}(\mu_0)\cup\mathcal{B}(\mu_1)$ has full $m$-measure. 

\end{lema}

\proof Let $Z=M\setminus (\mathcal{B}(\mu_0)\cup\mathcal{B}(\mu_1))$. Assume that $m(Z)>0$. Then, by Fubini's Theorem, we can find $\gamma_t=\gamma\times \{t\}$, where $\gamma$ is a closed disc in $\Td$ of radius less than $r_0$ and $t\in(0,1)$, such that $\pi(\gamma_t\cap Z)$ has positive $\nu$-measure. Then, for every $0<L<r_0$, it follows from Lemma~\ref{le:baciatotal2} that there exists a sequence of balls $\sigma_n\subseteq\Td$, $n\geq 1$, of radius $L$ such that
$$\lim_{n\to \infty}\dfrac{\nu\left( \sigma_n \cap E^n(\pi(\gamma_t\cap Z))\right)}{\nu\left(\sigma_n\right)}=1.$$
Recall that by construction, for each $n\geq 1$ there is an open set $\gamma_n\subseteq B(\theta_0,L)\subseteq \gamma$ (where $\theta_0$ is a $\nu$-density point of $\pi(\gamma_t\cap Z)$) such that $E^n(\gamma_n)=\sigma_n$. Since $\gamma_n\times\{t\}\subseteq \gamma_t$, we denote $\sigma^n_t=K^n(\gamma_n\times\{t\})$. Of course, we have that $\pi(\sigma^n_t)=\sigma_n$. With the notation above, we have
\begin{equation}\label{eq:dembaciatodo1}
\lim_{n\to \infty}\dfrac{\nu\left( \pi(\sigma^n_t) \cap E^n(\pi(\gamma_t\cap Z))\right)}{\nu\left(\sigma_n\right)}=1.
\end{equation}
Since $Z$ is $K$-invariant and $\pi$ is a semiconjugate, then for every $n\geq 1$, we have 
$$E^n\big(\pi(\gamma_t\cap Z)\big)=\pi\big(K^n(\gamma_t\cap Z)\big)=\pi\big(K^n(\gamma_t)\cap Z\big)$$
and since $\sigma^n_t\subseteq K^n(\gamma_t)$, then
$$\pi\big(\sigma^n_t\big) \cap E^n\big(\pi(\gamma_t\cap Z)\big)=\pi(\sigma^n_t\cap Z).$$
Therefore, the limit \eqref{eq:dembaciatodo1} implies that
\begin{equation}\label{eq:dembaciatodo2}
\lim_{n\to\infty}\nu\bigg(\pi\big(\sigma^n_t\cap [\mathcal B_0\cup\mathcal B_1]\big)\bigg)=0.
\end{equation}
On the other hand, note that for every $n\geq 1$, $\sigma^n_t$ is the graph of a $C^1$ function defined on $\sigma_n$. By Arzela-Ascoli, there exists a subsequence $\sigma_{n_k}$ converging to some disc $\sigma_\infty$ and the graphs $\sigma^{n_k}_t$ converge to the $C^1$-graph $\sigma^\infty_t$. Reducing the domains if necessary, we can assume that for every large $k\geq 1$, $\sigma_{n_k}=\sigma_\infty$. 

By Lemma~\ref{le:curvasc1}, there exists $\delta=\delta(\sigma_\infty)>0$ such that
$$ \nu\bigg(\pi\big(\sigma^\infty_t\cap[\mathcal B(\mu_0)\cup\mathcal B(\mu_1)]\big)\bigg)\geq \delta. $$
Moreover, the same lemma states that there exists $\epsilon>0$ such that for every sufficiently large $k\geq 1$, $\sigma_{n_k}$ is $\epsilon$-close of $\sigma_\infty$ in the $C^1$-topology, 
\begin{equation*}\label{eq:demlema2.10.1}
\nu\bigg(\pi\big(\sigma^{n_k}_t\cap[\mathcal B(\mu_0)\cup\mathcal B(\mu_1)]\big)\bigg)\geq \delta/2.
\end{equation*}
contradicting \eqref{eq:dembaciatodo2}.
\endproof

\subsection{Proof of Theorem~\ref{mainteo:KE1} part \ref{TeoAi} }\label{ssec:mmei} 

Consider the invariant probability measures $\mu_j=\nu\times \delta_j$, $j=0,1$. It is immediately clear that they are invariant and ergodic, respectively supported on the torus $\Td\times\{j\}$. Lemma~\ref{le:expnentenegativo} implies that each $\mu_j$ has negative center Lyapunov exponents. From Lemma~\ref{le:existenfisicas}, it follows that they are also observable measures with respect to $m$. Their supports are the whole $\Td\times\{j\}$, respectively, if $\nu$ is assumed to be fully supported on $\Td$. As a consequence, their basins are intermingled by Lemma~\ref{le:baciasintermingled}. Lemma~\ref{le:baciasestodo} implies that the union of their basins covers up to an $m$ measure of points of the cylinder $M$. Finally, it remains to be proven that there is no other ergodic $K$-invariant measure $\mu$ such that $\pi_*\mu=\nu$ with a negative center Lyapunov exponent.
\begin{lema}\label{le:characterization2} 
If $\mu\in M^{erg}(K)$ is such that $ \pi_* \mu = \nu$ and $\mu\notin
\{ \mu_0,\mu_1\}$, then $\lambda^c(K,\mu)\geq 0$.
\end{lema}

\proof 
If $\lambda^c(K,\mu)<0$, then $\nu$-almost every $\theta\in\Td$, there is a Pesin stable manifold that contains an interval $W^s(\theta)$ inside $I_\theta$. We can assume $W^s(\theta)$ is contained in the basin of $\mu$, and so $\mu$ has a basin with positive $m$-measure on $M$. It follows that $\mu=\mu_0$ or $\mu=\mu_1$ because the basins are disjoint and since by Lemma~\ref{le:baciasestodo}, their union covers $m$-almost every point of the whole manifold $M$.

\endproof

Therefore, we have proven the first statement in Theorem~\ref{mainteo:KE1}.


\section{Central Measure}\label{sec:mmeii}
Throughout this section we assume $\nu\in M^{erg}(E)$ such that $J_\nu E$ is H\"older continuous. We are interested in the set $M^{erg}_\nu(K)$ of measures $\mu\in M^{erg}(K)$ such that $\pi_*\mu=\nu$. 


\subsection{Separating map and separating measure}\label{ssec:separatingmap} 
Since $\mu_j=\nu\times\delta_j\in M^{erg}_\nu(K)$ and they have central Lyapunov exponent $\lambda^c(\mu_j)<0$, we can define a \textit{separating map} $\sigma:\Td\to M$ as following the ideas of Bonifant and Milnor \cite{BM2008} and Rodr\'iguez-Herz, Rodr\'iguez-Herz, Tahzibi, and Ures \cite{RRTU2012}. Denote by $\sigma_j$ a measurable map $\sigma_j: \Td \to I$ such that $\sigma_j(\theta)$ is the boundary of the connected component of the Pesin stable manifold of $(\theta,j)$. Each map $\sigma_j$ is well-defined for every $\theta\in\Td$ and they are measurable. Note that from Lemma~\ref{lem:intervalos}, $\sigma_j(\theta)\ne j$ for $\nu$-almost every $\theta\in \Td$. Moreover,

\begin{lema}\label{le:uniquenessseparating} $\sigma_0(\theta)=\sigma_1(\theta)$ for $\nu$-almost every point $\theta\in\Td$. 
\end{lema}

\proof Suppose that $\sigma_0(\theta)<\sigma_1(\theta)$ for a positive $\nu$-measure set of points $\theta\in\Td$. Then, Lemma~\ref{lem:intervalos} implies the existence of an open interval $J_\theta\subset I_\theta$ such that $\mathcal{B}(\mu_0)\cap J_\theta=\emptyset$ and $\mathcal{B}(\mu_1)\cap J_\theta=\emptyset$. This contradicts the conclusion of Lemma~\ref{le:baciasestodo}. \endproof

Therefore, we define $\sigma:\Td\to I$ as $\sigma(\theta)=\sigma_j(\theta)$, which is well defined for $\nu$-almost every $\theta\in\Td$. It is necessary to emphasize the following:

\begin{lema}\label{le:notcontinuous} 
The separating map $\sigma:\Td\to M$ is measurable but is not continuous.
\end{lema}

\proof If the map $\sigma$ is continuous, then the set $\{(\theta, \sigma(\theta) )\::\: \theta\in\Td\}$ separates the basins $\mathcal{B}(\mu_0)$ and $\mathcal{B}(\mu_1)$, and thus they cannot be intermingled.
\endproof

We point out that in the arguments above, we are invoking Lemma~\ref{le:baciasintermingled}, which requires $\supp \nu=\Td$ and Lemma~\ref{le:baciasestodo}, where the assumption of bounded distortion of $J_\nu E$  is required. Of course, the hypothesis that $J_\nu E$ is H\"older continuous implies that both conditions are satisfied.

Now, we will describe the conditional measures associated with Rokhlin decomposition for an ergodic measure projecting on $\nu$. It follows from Rokhlin's disintegration theorem that there exists a family of conditional probabilities $\{\mu_\theta\::\:\theta\in\Td\}$ with respect to the measurable partition $\{I_\theta\::\:\theta\in\Td\}$ such that
$$\int_M\Phi\,d\mu=\int_{\Td}\int_{I_\theta}\Phi\,d\mu_\theta\,d\nu(\theta),$$
\noindent for every $\Phi:M\to \mathbb{R}$ continuous. Note that in this case the quotient measure $\hat\mu$ coincides with $\nu$. 

Moreover, it follows from the definition of $\sigma:\Td\to I$ that, if $\mu\ne\mu_j$, then for $\nu$-almost every $\theta\in \Td$, $\mu_\theta=\delta_{(\theta,\sigma(\theta))}$ and $\supp\mu_\theta=\{(\theta,\sigma(\theta))\}$. In summary, as an immediate consequence of the existence of the separating map $\sigma:\Td\to I$ and the uniqueness of the Rokhlin decomposition, we have the following

\begin{lema}\label{le:separatingmeas} There exists a unique $\mu\in M^{erg}_\nu(K)$ such that $\mu\ne\mu_j$, $j=0,1$. It is defined by 
$$\int_M\Phi\,d\mu=\int_{\Td}\Phi(\theta,\sigma(\theta))\,d\nu(\theta)$$
for every $\Phi:M\to \mathbb{R}$ continuous.
\end{lema}

Note that the measure $\mu$ defined above is the push forward, under the section, $\theta\mapsto(\theta,\sigma(\theta))$, of the measure $\nu$. If we denote
$$\Sigma=\{(\theta,\sigma(\theta))\::\:\theta\in\Td\},$$
then $\Sigma$ is a $K$-invariant set. By the definition of $\mu$, for any Borelean set $A\subseteq M$,
$$\mu(A)=\nu(\pi(A\cap \Sigma))=\pi_*\nu(A\cap \Sigma).$$


\subsection{No measures having zero center exponent}\label{ssec:zerocenterexponent}

In this subsection, we prove that if $\mu\in M^{erg}_\nu(K)$ and $\mu\ne\mu_j$, then $\lambda^c(\mu)>0$.
From Lemma~\ref{le:characterization2}, we have that if $\mu\in M^{erg}_\nu(K)$, $\mu\ne\mu_j$, then $\lambda^c(\mu)\geq 0$. In the particular case $\lambda^c(\mu)=0$, we can state more about this decomposition.

\begin{mainlemma}\label{le:klemma}
Let $\nu\in M^{erg}(E)$ such that $J_\nu E$ is H\"older continuous and suppose that $\mu\in M^{erg}_\nu(K)$. Let $\{\mu_\theta\::\:\theta\in\Td\}$ be the family of conditional measures with respect to the partition $\{I_\theta=\{\theta\}\times I\::\:\theta\in \Td\}$. If $\lambda^c(\mu)=0$, then the map $\Td\ni\theta\to\mu_\theta\in M^1(I_\theta)$ is continuous.
\end{mainlemma}

This lemma follows from Avila-Viana's invariance principle \cite{AV2010}, and we proved it in Section~\ref{sec:Klemma}. For now, we use the lemma above to prove that there are no measures projecting on $\nu$ having a zero center Lyapunov exponent.

\begin{lema}\label{le:mmenonzero} Assume $\nu\in M^{erg}(E)$ has a H\"older continuous Jacobian. Let $\mu\in M^{erg}_\nu(K)$ be a measure such that $\mu \not \in \{ \mu_1, \mu_2\}$. Then, $\lambda^c(\mu)>0$.
\end{lema}

\proof  If $\lambda^c(\mu)=0$, then   Lemma~\ref{le:klemma} implies that $\sigma$ should be continuous contradicting Lemma~\ref{le:notcontinuous}.
\endproof


\subsection{Approximation by periodic points}\label{ssec:positivecase}

It is known in the context of expanding transformations on compact manifolds that any invariant probability measure $ \nu $ can be approximated in the weak* topology by invariant probability measures supported on periodic orbits (see Lemma~\ref{le:periodicpoints1}). Following the previous idea, we will establish the existence of a sequence of invariant probability measures supported at periodic points belonging to the interior of the cylinder approximating the separating measure $\mu$.

\begin{lema}\label{le:existence}
There exists a sequence of $K$-invariant probability measures supported on periodic points in the interior of $M$ converging to $\mu$ in the weak* topology. 
\end{lema}

\proof 
From Lemma~\ref{le:periodicpoints1}, $\nu$ can be approximated in the weak* topology by invariant probability measures supported in periodic points defined by
$$\delta\theta_n:=\frac1n\sum_{j=0}^{n-1} \delta_{E^j(\theta_n)},$$
\noindent where $\theta_n\in\Td$ is a periodic point of period $n$. That means that for every $\epsilon>0$, there exists an integer $N\geq 1$ such that for every $n\geq N$, there exists $\theta_n \in \Td$, a periodic point of period $n$ such that 
$$\left|\int_{\Td} \psi\, d\delta\theta_n-\int_{\Td} \psi\,d\nu(\theta)\right|<\epsilon,$$
\noindent for every $\psi:\Td\to\mathbb{R}$ continuous function. Using the inequality above for 
$$\psi(\theta)=\log|\partial_t \varphi (\theta,j)|, \quad j=0,1;$$
\noindent and passing to a subsequence if it is necessary, we conclude from \eqref{eq:condK4mme} that for every sufficiently large $n\geq 1$
$$\lambda^c(\delta\theta_n\times\delta_j)=\int_{\Td} \log|\partial_t \varphi (\theta,j)| d\,\delta\theta_n<0.$$
For every $\theta\in\Td$, let us denote $I_\theta=\{\theta\}\times I$. Now, since $\theta_n \in \Td$ is a periodic point of period $n$, then the map $\varphi_n:I_{\theta_n}\to I_{\theta_n}$ defined by
$$\varphi_n(t):=K^n(\theta_n,t)$$
\noindent fixes the extremes, and by the chain rule, we have
$$\int_{\Td} \log|\partial_t \varphi (\theta,j)| d\delta\theta_n = \frac1n \sum_{k=0}^{n-1} \log |\partial_t \varphi_k (j)| = \frac1n \log \prod_{k=0}^{n-1} |\varphi'_k(j)|= \frac1n\log |\varphi_n' (j)| < 0.$$
Thus, $|\varphi'_n(j)|<1$, $j=0,1$.  This implies that there exists $t_n\in(0,1)$, a fixed point of $\varphi_n$ and $|\varphi_n'(\theta_n,t_n)|\geq 1$. Therefore, $(\theta_n,t_n)$ is a periodic point of $K$ of period $n$ belonging to the interior of $M$. Moreover, by the definition of the separating map $\sigma$, we have $t_n=\sigma(\theta_n)$.

Finally, consider the sequence of $K$-invariant measures supported on the periodic point $(\theta_n,\sigma(\theta_n))$ defined by
\begin{equation*}\label{eq:permeas}
\mu_n=\frac1n\sum_{j=0}^{n-1} \delta_{K^j(\theta_n,\sigma(\theta_n))}.
\end{equation*}
Let $\tilde\mu$ be an ergodic component of any accumulation point of the periodic measures $\mu_n$. Thus,
$\delta\theta_n=\pi_*\tilde\mu_n\to\pi_*\tilde\mu$ as $n\to\infty$, and since we assume $\delta\theta_n\to\nu$, then $\pi_*\tilde\mu=\nu$. 
Moreover,
$$\lambda^c(\tilde\mu)=\int \log|\partial_t \varphi (\theta,t)| d\tilde\mu=\lim_{n\to\infty} \int \log|\partial_t \varphi (\theta,t)| d\mu_n=\lim_{n\to \infty}\lambda^c(\mu_n)\geq 0.$$
Then $\tilde\mu\in M_\nu^{erg}(M)$ and $\tilde\mu\ne\mu_j$, so from Lemma~\ref{le:separatingmeas} we conclude that $\tilde\mu=\mu$. 
\endproof


\subsection{Measures with full support}\label{ssec:mfull}
Now we prove that the measure $\mu$ obtained in the previous sections is fully supported. We will need the following result due to Gan and Shi
 \cite{GS2019}, which implies that Kan's map is robustly transitive. 

\begin{lema}{\cite[Theorem 2]{GS2019}}\label{le:transitive}
Letting $K$ be a $C^2$ Kan-like map satisfying \ref{K1}, \ref{K2} and \ref{K3}, then for any nonempty open subset $U\subseteq M$, we have 
\begin{equation}\label{eq:robtrans}
    {\rm Int} (M)\subseteq\bigcup_{n\geq 0} K^n(U).
\end{equation}
Moreover, if $K$ is $2$-partially hyperbolic, there exists an $C^2$ open neighborhood $\mathcal U$ of partially hyperbolic endomorphisms preserving the boundary, such that every $F\in \mathcal U$ satisfies the property \eqref{eq:robtrans}.
\end{lema}

The original statement in \cite{GS2019} is formulated in a slightly less general setting than the one that we are considering.
Gan and Shi consider $C^2$-Kan maps $K\::\:\mathbb S^1\times I\to \mathbb S^1\times I$ defined by
$$K(\theta, t)=(d\cdot\theta,\varphi(\theta, t)), \quad |d|\geq 2.$$
Lemma~\ref{le:transitive} is the cornerstone of the work \cite{GS2019}, so we review briefly the proof in order to show why it works if we are considering any uniform expanding map $E$ defined on $\Td$ instead of the affine map on $\mathbb S^1$.

The proof includes two steps. In the first one, it is assumed that the center Lyapunov exponent of the fixed point $p,q\in \Td$ provided by \ref{K3} has ``bad'' rational dependence. More precisely, it is assumed that there exists $\varepsilon=\varepsilon(U)>0$, and that there exist integers $k,l>0$, satisfying \begin{equation}\label{eq:badrelation}
0<|k\cdot \lambda^c(p)+l\cdot \lambda^c(q)|<\varepsilon.
\end{equation}
This then allows us to conclude \eqref{eq:robtrans}. We point out that the classical Kan example \eqref{eq:Kanoriginal} satisfies the relation \eqref{eq:badrelation} above.

The core of the proof of the first step is to reduce the problem to 1-dimensional dynamics in the central direction. More precisely, we begin with an open set $U\subseteq \textrm{int} (M)$. Note that every open set $U\subseteq \textrm{int}(M)$ eventually intersects $W^s((p,0)$ and $W^s(q,1)$ up to some iteration. If $U$ (or some future iteration) intersects $W^s(p,0)$ in an open interval $J_p$ containing a fundamental domain for the dynamics $\varphi_p\::\:I\to I$, then due to the partial hyperbolicity (see \ref{K2}) of $K$ and the existence of a heteroclinic cycle relating the boundaries by the dynamics on $W^s(p,0)$ and $W^s(q,1)$ (see \ref{K3}), we can conclude \eqref{eq:robtrans}. The same conclusion can be deduced if $U$ (or some future iteration) intersects $W^s(q,1)$ in an open interval $J_q$ containing a fundamental domain for the dynamic $\varphi_q$. If $\cup_{n\geq 1}\varphi^n_p(J_p)$ does not contain a fundamental domain, we can try to link $J_p$ with an interval $J_q$ by the unstable holonomy (in the lifting), iterate $J_q$ by $\varphi_q$, and return back to $W^s(p,0)$ again by the unstable holonomy to an interval $J^1_p$, and we consider
$$\left(\cup_{n\geq 1}\varphi^n_p(J_p)\right)\cup \left(\cup_{n\geq 1}\varphi^n_p(J^1_p)\right).$$
If the set above contains a fundamental domain, we are finished. Else, we can continue as before. The algebraic condition \eqref{eq:badrelation} prevents the existence of points in $\{p\}\times (0,1)$, which are not covered with this enumerable process. Of course, Sternberg's linearization theorem applied to the central dynamics $\varphi_p$ and $\varphi_q$ plays a key role for this argument. All of the above construction is independent of whether the manifold at the base is $\S$ or $\Td$, but the key fact is that the dynamic is uniformly expanded. Therefore, following the ideas in \cite{GS2019}, it is straightforward to extend this case to our setting.  

In the second step, the authors present the case when the center Lyapunov exponents of $p$ and $q$ have a rational dependence, meaning that they do not satisfy \eqref{eq:badrelation}. To do this, the following dichotomy is proven: 

\begin{enumerate}
\item Either there exists a sequence of periodic points whose central Lyapunov exponents are asymptotically rationally independent of $\lambda^c(q)$ and that the maps in those fibers are of the type South-North; or
\item there is an explicit expression (coboundary) for the Lyapunov exponent of any point in $\Td\times \{0\}$ from the Lyapunov exponent of $(q,0)$.
\end{enumerate}
Gan and Shi noted that if we evaluate the fixed point $(p,0)$ in the expression given in the second part of the dichotomy, we obtain a contradiction with \ref{K3}. Therefore, (1) must be true and then we can apply step 1. 

Note again that this does not depend on the base manifold being $\S$. The last part of the arguments allows the conclusion that the relation \eqref{eq:robtrans} holds for every $C^2$ boundary preserving map which is sufficiently $C^1$-close to $K$.

\begin{lema}\label{le:soportetotalsep}
$\supp \mu=M$
\end{lema}

\begin{proof} First we prove that for every nonempty open subset $V\subseteq \mbox{Int(M)}$, if $\mu(V)=0$, then $\mu(K(V))=0$. 

Indeed, suppose by the contrary that for any nonempty open subset $V\subseteq \mbox{Int}(M)$, $\mu(K(V))>0$. Therefore,
\begin{eqnarray*}
0&<&\mu(K(V))=\nu_*\pi(K(V)\cap\Sigma)\\
&=&\nu(\pi\circ K(V\cap \Sigma))\\
&=&\nu(E\circ\pi(V\cap \Sigma))\\ 
&=&\int_{\pi(V\cap\Sigma)}J_\nu E\,d\nu.
\end{eqnarray*}
Hence,
$$0<\nu(\pi(V\cap\Sigma))=\mu(V).$$
Now, suppose that there exists a nonempty open subset $U\subseteq M$ such that $\mu(U)=0$. By the Lemma~\ref{le:transitive}, we have that $\cup_{n\leq 0} K^n(U)=\mbox{Int}(M)$, then by the previous claim we have that $\mu(\mbox{Int}(M))=0$: this contradicts the construction of $\mu$.

\end{proof}


\subsection{Proof of Theorem~\ref{mainteo:KE1} part \ref{TeoAii}}\label{ssec:mmeii}
Finally, we summarize the proof of Theorem~\ref{mainteo:KE1} part \ref{TeoAii}. More precisely, we have proven that

\begin{prop}\label{teo:medidapositiva}
Assume $K$ is a Kan-like map and suppose that $\nu$ is any ergodic $E$-invariant probability measure with $\supp \nu=\Td$ and that $J_\nu E$ is H\"older continuous. Assume that condition \eqref{eq:condK4im}:
$$\int\log |\partial_t \varphi (\theta,j)|\: d\nu(\theta)<0, \quad j=0,1$$
is verified. Then:

\begin{enumerate}
        \item[(i)] There is a measurable, not continuous map $\sigma: \Td \to I$ such that for $\nu$-a.e. $\theta\in \Td$ 
        $$W^s(\theta, 0)=\{\theta\} \times [0, \sigma(\theta)),\quad\mbox{ and }\quad W^s(\theta, 1)=\{\theta\} \times (\sigma(\theta), 1].$$ 
        
       \item[(ii)] There exists a unique ergodic $K$-invariant probability measure $\mu$ such that $\pi_*\mu=\nu$ and $\mu\ne\mu_j$, $j=0,1$. Such measure $\mu$ is defined by
$$\int_M\Phi\,d\mu=\int_{\Td}\Phi(\theta,\sigma(\theta))\,d\nu(\theta),$$
for every $\Phi:M\to \mathbb{R}$ continuous.

        \item[(iii)] Moreover, $\mu$ has a positive central Lyapunov exponent, it is approximated by measures supported along periodic orbits in the interior of $M$, and $\supp\mu=M$.
    \end{enumerate}
\end{prop}

\proof  The first and second items summarize the results obtained in  Subsection~\ref{ssec:separatingmap}. Item iii part (a) is proved in Lemma~\ref{le:mmenonzero}. Lemma~\ref{le:existence} provides a proof of item (iii) part (b). Finally, item (iii) part (c) corresponds to Lemma~\ref{le:soportetotalsep}.
\endproof

With this proposition, we complete the proof of Theorem~\ref{mainteo:KE1}.


\section{Invariance principle for inverse limits}\label{sec:Klemma}

This section is devoted to the proof of Lemma~\ref{le:klemma}: for such a purpose, we use Avila-Viana's invariance principle \cite{AV2010} applied to the natural extension of the map $K$ to its inverse limit.
The presentation of this section closely follows the presentations developed in \cite{QXZ2009,VY2019}. 

\subsection{Inverse Limits}\label{ssec:inverselim}

Define $\hS$ to be the space of all sequences $\htheta=(\theta_{-n})_{n\geq 0}$ such that $\theta_i \in \Td$ and $E(\theta_{-n}) = \theta_{-n+1}$ for all $n\geq 1$. In $\hS$, consider the metric 
$$\dist_{\hS}(\halpha,\hbeta)=\sum_{n\geq 0}\lambda^n\dist_{\Td}(\alpha_{-n},\beta_{-n}),$$
where $0<\lambda<1$, $\halpha=(\alpha_{-n})_{n\geq 0}$, $\hbeta=(\beta_{-n})_{n\geq0}$. It is straightforward from the definition that $\hS$ is a compact metric space. 

The \emph{natural extension} of $E$ is the map $\hE:\hS\to\hS$ defined by
$$\hE(\dots,\theta_{-n},\dots,\theta_{-1},\theta_0)=(\dots\theta_{-n},\dots,\theta_{-1},\theta_0, E(\theta_0)).$$ 
Then, $\hE$ is a hyperbolic homeomorphism and satisfies $\hpi \circ \hE =E \circ \hpi$, where $\hpi :\hS\to \Td$ is the canonical projection to the zeroth term. 

Every point $\halpha\in\hS$ has defined (locally) stable and (locally) unstable sets: for any $\epsilon>0$ and with $\halpha\in\hS$ fixed, we denote them by 
\begin{eqnarray*}
W_\epsilon^s(\halpha)&=&\{\hbeta\in\hS\::\: \dist_{\hS}(\hE^n(\halpha),\hE^n(\hbeta))\leq\epsilon \mbox{ for all }n \geq 0\}, \mbox{ and }\\
W_\epsilon^u(\halpha)&=&\{\hbeta\in\hS\::\: \dist_{\hS}(\hE^n(\halpha),\hE^n(\hbeta))\leq\epsilon \mbox{ for all }n \leq 0\}.
\end{eqnarray*}
Then, the \textit{stable} and \textit{unstable} sets of $\halpha$ are given by
$$W^s(\halpha)=\bigcup_{n\geq 0} \hE^{-n}(W^s_\epsilon(\hE^n(\halpha))\quad\mbox{ and }\quad W^u(\halpha)=\bigcup_{n\geq 0} \hE^{n}(W^u_\epsilon(\hE^{-n}(\halpha)).$$
Assuming $\epsilon>0$ is small enough, $W_\epsilon^s(\htheta)$ coincides with the fiber $\hpi^{-1}(\theta_0)$ and $\hpi$ maps $W^u_\epsilon (\htheta)$ homeomorphically to
$D_{\htheta} = \hpi( W^u_\epsilon ( \htheta ) ),$
with 
$$B(\theta_0, 9\epsilon/10)\subset D_{\htheta}\subset B(\theta_0, \epsilon).$$
Moreover, each $\hat{B}_{\htheta}= \hpi^{-1}(D_{\htheta})$ may be identified with the product (box)
\begin{equation}\label{eq:caja}
\hat{B}_{\htheta}:= D_{\htheta} \times \hat\pi^{-1}(\theta_0) \approx W^u_\epsilon(\htheta) \times W^s_\epsilon( \htheta ).
\end{equation}
through a homeomorphism, so that $\hpi$ becomes the projection to the first coordinate.

It is well known \cite[Proposition I.3.1]{QXZ2009} that for every $E$-invariant probability measure $\nu$, there exists a unique $\hE$-invariant probability measure $\hat{\nu}$ such that $\hpi_*\hat{\nu}=\nu$. Moreover, if $\nu$ has a H\"older continuous Jacobian for $E$, then $\hnu$ has local product structure \cite{B1975}, \cite[Section 2.2]{BV2004}: the restriction of $\hnu$ to each box $\hat{B}_{\htheta}$ may be written as
\begin{equation}\label{eq:prodloc}
\hnu|\hat{B}_{\htheta} = \xi(\hnu^u \times\hnu^s ),
\end{equation}
where $\xi:\hat{B}_{\htheta}\to (0,\infty)$ is a continuous function, $\hnu^u=\nu|D_{\htheta}$ and $\hnu^s$ is a probability measure on $W^s_\epsilon(\htheta)$ equivalent to the Bernoulli measure associated with the shift map defined in the space of the sequence of $k$-symbols, $k=\deg E$. This means that $\hnu | \hat{B}_{\htheta}$ is equivalent to a product $\hnu^u \times \hnu^s$, with $\xi$ as the Radon-Nikodym density.

\subsection{Smooth cocycle}\label{ssec:cocycles}

A \textit{smooth cocycle} over a hyperbolic homeomorphism is a cocycle that acts by diffeomorphisms on the fibers. For more details, see \cite[Secction 2.1]{AV2010}.

From subsection~\ref{ssec:inverselim}, we can induce a natural invertible smooth cocycle $\hK:\hS\times I \to\hS\times I$, defined by 
$$\hK(\htheta,t)=(\hE(\htheta),\varphi(\theta_0,t)).$$ 
Notice that the second coordinate of $\hK$ depends only on $\theta_0$, and then $\hK$ is a \textit{locally constant skew product} (see \cite[Section 5.1]{V2014}).

Again, we have $(\hpi\times Id) \circ \hK =K \circ (\hpi\times Id)$, and for every $K$-invariant probability measure $\eta$, there exists a unique $\hK$-invariant probability measure $\hat\eta$ such that $(\hpi\times Id)_*\hat\eta=\eta$. 

Let $\hat{p}:\hS\times I\to I$ be the projection $\hat{p}(\htheta,t)=t$. For $n\in\mathbb{Z}$, we denote by $\hK^n_{\htheta}=I_{\htheta}\to I_{\hE^n(\htheta)}$ the map defined by
$$\hK^n_{\htheta}(t):=\hat{p}\circ \hK^n(\htheta,t).$$
The \textit{Lyapunov exponent} of $\hK$ at a point $(\htheta,t)\in \hS\times I$ is defined by
$$\lambda(\hK,\htheta,t)= \lim_{n\to+\infty} \frac1n \log|D\hK^n_{\htheta} (t)|.$$
If $p:\mathbb{T}^d\times I\to I$ is the projection $p(\theta,t)=t$, then we have
$$\hK^n_{\htheta}(t)=\hat{p}\circ \hK^n(\htheta,t)=p\circ K^n(\theta_0,t)=K^n_{\theta_0}(t)$$
and so
$$\lambda(\hK,\htheta,t)=\lambda^c(K, \theta_0,t),$$
where $\lambda^c(K, \theta_0,t)$ is a center Lyapunov exponent of $K$ for the point $(\theta_0,t)$. 

If $\hat\eta$ is an $\hK$-invariant probability measure, then the \textit{integrated Lyapunov exponent} of $\hK$ with respect to $\hat\eta$ is 
$$ \lambda(\hat{K},\hat\eta)=\int \lambda(\hK,\htheta,t)\,d\hat\eta.$$
Obviously, if $\hat\eta$ is ergodic, then $\hat\eta$-almost every point $(\hat \theta,t)\in\hS\times I$,
$$\lambda(\hK,\htheta,t)= \lambda(\hat{K},\hat\eta) .$$
In particular, if $\eta\in M^1(K)$ and $\hat\eta\in M^1(\hat{K})$ are such that $(\hpi\times Id)_*\hat\eta=\eta$, then 
$$\lambda(\hat{K},\hat\eta)=\int \lambda(\hK,\htheta,t)\,d\hat\eta=\int \lambda^c(K,\theta,t)\,d\eta=\lambda^c(K, \eta).$$
In summary, as a direct consequence of the discussion above, we have the following lemma.

\begin{lema}\label{le:expolyapigual} If $\mu\in M^{erg}(K)$ is such that $\lambda^c(K, \mu)=0$, then there exists a unique $\hmu\in M^{erg}(\hK)$ such that $(\hpi \times Id)_*\hmu=\mu$ and $\lambda(\hat{K},\hat\mu)=0$.
\end{lema}

\subsection{Disintegration revisited}\label{ssec:desintrev}

The goal of this section is to study the relations in the disintegration associated with the invariant measure of $K$ and $\hK$.

\begin{lema}\label{le:relmedcond} Let $\mu\in M^1(K)$ and $\hmu\in M^1(\hK)$ be such that 
\begin{equation}\label{eq:4.2.0}
(\hpi\times Id)_*\hmu=\mu.
\end{equation}
Let $\{\hmu_{\htheta}\::\:\htheta\in\hS\}$ and $\hnu$, respectively, be the family of conditional measures and the quotient measure of $\hmu$ with respect to the measurable partition $\{I_{\htheta}:=\{\htheta\}\times I\::\:\htheta\in\hS\}$. Let $\{\mu_{\theta}\::\:\theta\in\Td\}$ and $\nu$, respectively, be the family of conditional measures and the quotient measure of $\mu$ with respect to the measurable partition $\{I_{\theta}:=\{\theta\}\times I\::\:\theta\in\Td\}$.
Then, for every continuous function $\psi\in C^0(\Td\times I,\mathbb{R})$, we have
\begin{equation*}\label{eq:4.2}
\int_{I_{\theta}}\psi(\theta,t)\,d\mu_\theta(t)=\int_{\hpi^{-1}(\theta)}\xi(\htheta)\int_{I_{\htheta}}\psi(\theta,t)\,d\hmu_{\htheta}(t)\,d\hnu^s(\htheta).
\end{equation*}
\end{lema}

\proof The relation \eqref{eq:4.2.0} above is equivalent to 
\begin{equation}\label{eq:4.2.1}
\int_{\Td\times I}\psi(\theta,t)\,d\mu= \int_{\hS\times I}\psi(\pi(\htheta),t)\,d\hmu,
\end{equation}
for every $\psi\in C^0(\Td\times I,\mathbb{R})$. By the disintegration of $\hmu$ with respect to the partition $\{I_{\htheta}\::\:\htheta\in\hS\}$, we have that
\begin{equation}\label{eq:4.2.2}
\int_{\hS\times I}\psi(\pi(\htheta),t)\,d\hmu= \int_{\hS}\int_{I_{\htheta}}\psi(\pi(\htheta),t)\,d\hmu_{\htheta}\,d\hnu(\htheta),
\end{equation}
where $\hnu$ is the unique $\hat{E}$-invariant probability measure supported on $\hat{S}$ such that $\hat{p}_{1_*}\hmu=\hnu$ (here, $\hat{p}_1:\hS\times I\to\hS$ denotes the projection in the first coordinate $\hat{p}_1(\htheta,t)=\htheta$).

Assume that $\psi$ is supported on $\hat{B}\times I$, where $\hat{B}$ is a box defined as in \eqref{eq:caja}. By combining \eqref{eq:prodloc} and \eqref{eq:4.2.2} and taking into account that $\hnu^u=\nu|D$, where $D\subseteq\Td$ is a disc such that $\hat{B}=\hpi^{-1}(D)$, we obtain
\begin{equation}\label{eq:4.2.3}
\int_{\hat{B}\times I}\psi(\pi(\htheta),t)\,d\hmu= \int_{\hat{B}}\xi(\htheta)\int_{I_{\htheta}}\psi(\pi(\htheta),t)\,d\hmu_{\htheta}\,d\hnu^s(\htheta)\,d\nu(\theta),
\end{equation}
with $\theta_0=\theta$.
On the other hand, from \eqref{eq:4.2.1} and the disintegration of $\mu$, we have
\begin{equation}\label{eq:4.2.4}
\int_{\hS\times I}\psi(\pi(\htheta),t)\,d\hmu=\int_{\Td\times I}\psi(\theta,t)\,d\mu=\int_{\Td}\int_{I_{\theta}}\psi(\theta,t)\,d\mu_\theta(t)\,d\nu(\theta).
\end{equation}
thus, from \eqref{eq:4.2.3} and \eqref{eq:4.2.4} we obtain that
$$
\int_{I_{\theta}}\psi(\theta,t)\,d\mu_\theta(t)=\int_{\hpi^{-1}(\theta)}\xi(\htheta)\int_{I_{\htheta}}\psi(\theta,t)\,d\hmu_{\htheta}(t)\,d\hnu^s(\htheta),
$$
which is the relation we sought.
\endproof


\subsection{Invariance Principle}\label{ssec:invprin}

In this section, we will present the invariance principle of Avila-Viana, formulated in the context of the invertible smooth cocycle $\hK$ over $\hE$.

An \textit{s-holonomy} for $\hK$ is a family $h^s$ of H\"older homeomorphisms $h^s_{\halpha,\hbeta}: I_{\halpha}\to I_{\hbeta}$ defined for every $\hbeta\in W^s(\halpha)$, satisfying

\begin{itemize}
\item[{(sh1)}] $h^s_{\halpha,\hgamma} \circ h^s_{\hgamma,\hbeta} =h^s_{\halpha,\hbeta}$ and $h^s_{\htheta,\htheta} =id$.
\item[{(sh2)}] $\hK_{\hbeta} \circ h^s_{\halpha,\hbeta} = h^s_{\hE(\halpha),\hE(\hbeta)} \circ \hK_{\halpha}.$

\item[{(sh3)}] $(\halpha,\hbeta,t)\to h^s_{\halpha,\hbeta} (t)$ is continuous. 
\end{itemize}

In the last condition, $(\halpha,\hbeta)$ varies with respect to the space of pairs of points in the same local stable set. A disintegration $\{\hmu_{\htheta} : \htheta \in \hS\}$ of the $\hK$-invariant probability measure $\hmu$ is \textit{s-invariant} if for each $\hbeta \in W^s(\halpha)$, we have
$$(h^s_{\halpha,\hbeta})_*\hmu_{\halpha} =\hmu_{\hbeta},\qquad \mbox{ for every } \hbeta \in W^s(\halpha)$$ 
with $\halpha$ and $\hbeta$ in support of the projection $\hnu=\hat{p}_{1_*}\hmu$. 

By replacing $\hE$ and $\hK$ by their inverses, one obtains dual notions of $u$-holonomy, $h^u $ and $u$-invariant disintegration. Since $\hK$ is locally constant along the stable direction, then $\hK$ admits a trivial stable holonomy \cite[Section 10.6]{V2014}:
$$h^s_{\halpha,\hbeta}=Id,$$
for every $\halpha,\hbeta$ in the same stable set. On the other hand, the existence of the unstable holonomy is provided by the holonomy associated with the unstable foliation. 

Then, the invariance principle of Avila-Viana \cite[Theorem D]{AV2010} can be reformulated in our setting as follows.

\begin{teo}[Invariance Principle]\label{teo:princinv}

Assume that $\hat K$ admits an s-holonomy and a u-holonomy. 
Let $\hmu$ be a $\hK$-invariant probability measure such that $\hat{p}_{1_*}\hmu=\hnu\in M^1(\hE)$ admits a local product structure. Assume that $\lambda(\hat{K},\hat\mu)= 0$. Then, $\hmu$ admits a disintegration $\{\hmu_{\htheta} : \htheta \in \hS\}$ that is $s$-invariant and $u$-invariant, and their conditional probability measures $\hmu_{\htheta}$ vary continuously with $\htheta$ on the support of $\hat\nu$.
\end{teo}

With this last tool, we can conclude the proof of Lemma~\ref{le:klemma}.

\proof[{Proof of Lemma~\ref{le:klemma}}] Let $\mu\in M^{erg}_0(K)$ such that $\lambda^c(K,\mu)=0$. Let $\hmu$ be the unique $\hK$-invariant probability measure such that
$$(\hpi\times Id)_*\hat{\mu}=\mu.$$ 
Then, $\hat{p}_{1_*}\hmu=\hnu\in M^1(\hE)$ admits a local product structure. Moreover, it follows from Lemma~\ref{le:expolyapigual} that $\lambda(\hat{K},\hat\mu)=0$. From the invariance principle, we conclude that the disintegration $\{\hmu_{\htheta} : \htheta \in \hS\}$ (which is unique by Rokhlin's Theorem) is $s$-invariant and $u$-invariant and the conditional probabilities $\hmu_{\htheta}$ vary continuously with $\htheta$ on the support of $\hat\nu$. However, it follows from Lemma~\ref{le:relmedcond} that there exist $\xi:\hS\to \mathbb{R}$ continuous, such that for every function $\psi\in C^0(\S\times I,\mathbb{R})$, we have
$$\int_{I_{\theta}}\psi(\theta,t)\,d\mu_\theta(t)=\int_{\hpi^{-1}(\theta)}\xi(\htheta)\int_{I_{\htheta}}\psi(\theta,t)\,d\hmu_{\htheta}(t)\,d\hnu^s(\htheta).$$
From the relation above, it follows that the conditional probability measures $\{\mu_{\theta}\::\:\theta\in\Td\}$ of $\mu$ with respect to the partition $\{I_\theta\::\:\theta\in\Td\}$ vary continuously with $\theta$ in the support of $\nu$, as we sought to prove.
\endproof


\section{Fibered maps}\label{sec:ProofThF}
We assume along the whole section that $F:M\to M$ is a $C^r$-local diffeomorphism, $r\geq 1$,  defined by
$$F(\theta,t)=(E(\theta,t), \varphi(\theta,t)).$$
Assume that $F$ satisfies \ref{F1}-\ref{F5}. 

Recall, we denote by $E_j:\Td\to\Td$ the expanding map defined by $E_j(\theta)=E(\theta, j)$, $j=0,1$. To simplify the notation, we also denote by $E_j$ the map $E_j:\Td\times\{j\}\to\Td\times\{j\}$ defined by $E_j(\theta,j)=(E(\theta,j),j)$. Therefore, the subindex $j=0,1$ indicates in which torus $\Td\times\{j\}$ we are working. Note that we have a natural identification among the (ergodic) $E_j$-invariant probability measures $\nu_j$ on $\Td$ and the (ergodic) $F$-invariant probability measures $\mu_j$ supported on the boundary $\Td\times\{j\}$ given by $\mu_j=\nu_j\times \delta_j$. 

Due to the existence of central foliation by \ref{F3}, there is a \textit{natural projection along the central leaves} $\pi_j:M\to \Td$ defined by $\pi_j(\theta,t) \in\Td$ such that $\mathcal F^c(\theta,t)=\mathcal F^c(\pi_j(\theta,t),j)$. In this section, we denote $j,j^*\in\{0,1\}$ such that $j+j^*=1$. 

Moreover, for $t,s\in I$, we can define the \textit{holonomy map} $\pi_{t,s}:\Td\times\{t\}\to\Td\times\{s\}$ as follows: If $\mathcal{F}^c(\theta, t)$ is the central curve passing by $(\theta, t)$, then $\pi_{t,s}(\theta,t)=(\hat{\theta},s)$, where $(\hat{\theta},s)$ is the unique point belonging to the intersection $\mathcal{F}^c(\theta,t)\cap (\Td\times\{s\})$. Condition \ref{F3} implies that $\pi_{t,s}$ is a bi-H\"older homeomorphism with $(\pi_{t,s})^{-1}=\pi_{s,t}$. 

In particular, the map $\pi_{0,1}\::\:\Td\times\{0\}\to\Td\times\{1\}$ is a topological conjugation among the expanding map $E_0$ defined on the bottom torus and the expanding map $E_1$ defined on the top. Hence, there is correspondence between the invariant measures of both expanding maps. A measure supported on the bottom boundary (resp. the top boundary) is fixed, and we will denote it with a Greek letter and subindex it by zero (resp. one). The same Greek letter with a different subindex will indicate that one is the image of the other according to the appropriate reversion of the holonomy: More precisely,
\begin{equation}\label{eq:imagenmed}
 \mu_j=(\pi_{j,j^*})^*\mu_{j^*}   
\end{equation}
Given a probability measure $\nu$ supported on the torus $\Td$, we define the probability measure $m_\nu^j$ on the cylinder $M$ induced by the central foliation as
\begin{equation}\label{eq:refmeas2}
    m_\nu^j(A)=\int (\pi_{t,j})^*\nu(A)\,d\,\m_{I}(t), \quad \mbox{for every }A\subseteq M \mbox{ measurable.}
\end{equation}
It follows from the definition that if $\nu_j$ is an $E_j$-invariant measure, and $\mu_j=\nu_j\times\delta_j$ satisfies \eqref{eq:imagenmed}, then
\eqref{eq:refmeas2} does not depend on $j=0,1$. 


\subsection{Measures supported on the boundaries.} We proceed as in the proof of Theorem~\ref{mainteo:KE1} to prove the following partial result. 

\begin{prop}\label{prop:medidasnegativasparaF}
Let $F:M\to M$ be a $C^r$-local diffeomorphism, $r\geq 1$, satisfying \ref{F1}-\ref{F5}. 
Assume $\mu_j=\nu_j\times\delta_j$, satisfying the relation \eqref{eq:imagenmed}, and let $m$ be the reference measure defined by \eqref{eq:refmeas2} with $\nu=\nu_j$. Then, 
    \begin{itemize}
        \item[(a)] the measures $\mu_j$, $j=0,1$, are observable with respect to $m$.
        \item[(b)] If $\supp \nu=\Td$, then the basins of $\mu_j$ are intermingled with respect to $m$.     
        \item[(c)] If in addition $E_j$ has a Jacobian with respect to $\nu_j$ with bounded distortion, then the union of their basins $\mathcal{B}(\mu_0)\cup \mathcal{B}(\mu_1)$ covers $m$-a.e. the whole manifold $M$.
        \item[(d)] There is no other $F$-invariant measure $\mu$ on $M$ such that $\pi_*\mu=\nu$ with a negative center Lyapunov exponent. 
\end{itemize}

\end{prop}

The ideas of the proof are essentially the same as Proposition~\ref{teo:medidasnegativas} before. In the following, we summarize the highlight of the proof by pointing out the main differences and how we outline the difficulties appearing in this setting.

Recall that the \textit{center Lyapunov exponent} of $\mu_j$ is 
$$\lambda^c(\mu_j)=\int\log |DF|E^c|\: d\,\mu_j,$$
and condition \ref{F5} is equivalent to having $\lambda^c(\mu_j)<0$ for $j=0,1$. As a consequence of the Pesin theory, for $\nu_j$-almost every $\theta\in\Td$, the stable manifold $W^s(\theta,j)$ on $(\theta,j)$ is a curve inside the leaf $\mathcal{F}^c(\theta,j)$ (just as in Lemma~\ref{lem:intervalos}). 

\subsubsection{$\mu_0$ and $\mu_1$ are observable measures with respect to the reference measure $m$.} 

The proof continues as in Lemma~\ref{le:existenfisicas}.  
 
Our assumption that $\mathbb{R}^d\times\{0\}\subseteq \mathcal{C}^u(\theta,t)$ in \ref{F2} implies that $\Td\times\{0\}$ is a submanifold transverse to $E^c_{(\theta,0)}$ for every $\theta\in\Td$. Since $\nu_j$ is assumed to be an ergodic $E_j$-invariant measure on $\Td$, then the ergodic $F$-invariant measure $\mu_j$ has support contained in $\Td\times\{j\}$. Just as in Lemma~\ref{le:existenfisicas} we conclude that
$$\mu_j((\Td\times\{j\})\cap\mathcal{B}(\mu_j)\cap\mathcal{R}(\mu_j))=1, \quad j=0,1.$$

In particular, for fixed and sufficiently small $\epsilon>0$, there exists a set $\Gamma_j\subseteq \Td$ such that $\nu_j(\Gamma_j)>0$ and for every $\theta\in \Gamma_j$, the Pesin stable manifold $W_\epsilon^s(\theta,j)\subseteq \mathcal F^c(\theta,j)$ is a $C^1$-curve (compare with Lemma~\ref{lem:intervalos}). We define the ``rectangle'' 
$$B_j(\epsilon):=\bigcup_{\theta\in\Gamma_j}W_\epsilon^s(\theta,j),\quad j=0,1.$$
Note that $W_\epsilon^s(\theta,j)\subseteq \mathcal B(\mu_j)$ and so $B_j(\epsilon)\subseteq B(\mu_j)$. 
By \ref{F3i}, for each leaf $\mathcal F^c(\theta, j)$ the center foliation is $C^1$-close enough to $\{\theta\}\times I$, and therefore the lengths of $W_\epsilon^s(\theta,j)$, $\theta\in\Gamma_j$ are comparable with a segment of length greater than $\epsilon$, and from the definition of $m$ (see \eqref{eq:refmeas2}) we have
$$m(\mathcal B(\mu_j))\geq m ( B_j(\epsilon))\geq \nu_j(\Gamma_j)\cdot\epsilon>0.$$

\subsubsection{$\mu_0$ and $\mu_1$ are intermingled with respect to $m$.} We now proceed as in Lemma~\ref{le:baciasintermingled}.

Take any disc $\gamma \subset \Td$ and $t \in (0,1)$. We denote $\gamma_t = \gamma \times \{ t\}$, the disc inside the open cylinder $\Td\times (0,1)$ transverse to the $E^c$-direction. From \ref{F2}, we have that the map $F$ is partially hyperbolic: the derivative in the $E^c$-direction is dominated by the derivative in the $\mathcal{C}^u$-cone field. Then, the size of $\gamma_t$ grows exponentially fast and the angle between (the tangent space of) $\gamma_t$ and the $E^c$-direction is bounded far away from $0$ as the disc $\gamma_t$ is successively iterated. Then, up to taking some forward iteration $n$ denoted by $\gamma_t^n:=F^n(\gamma_t)$, we have that $\gamma_t^n$ crosses (transversely) the invariant central curves $\mathcal F^c(p,0)$ and $\mathcal F^c(q,1)$. Since \ref{F4} determines the dynamics on such invariant curves as in the Kan-like maps, the orbit of $\gamma_t$ accumulates on both boundary components of the cylinder $\Td \times [0,1]$.

Given $\epsilon>0$, and taking a sufficiently large forward iterate, we can assume that $\gamma_t^n$ transversely intersects the rectangle $B_j(\epsilon)$ defined above and $\nu(\pi_j(\gamma_t^n\cap B_j(\epsilon)))>0$. 
Again, $\gamma_t^n\cap B_j(\epsilon)$ and $\pi_j(\gamma_t^n\cap B_j(\epsilon))$ are subsets of the $F$-invariant set $\mathcal B(\mu_j)$, and so $F^{-n}(\gamma_t^n\cap B_j(\epsilon))\cap \gamma_t \subseteq \mathcal B(\mu_j)$ and 
\begin{eqnarray*}
(\pi_{t,j})^*\nu_j\Bigg(\mathcal B(\mu_j)\cap \gamma_t \bigg) &\geq& \nu_j\Bigg(\pi_j\bigg(F^{-n}\big(\gamma_t^n\cap B_j(\epsilon)\big)\cap \gamma_t \bigg) \Bigg)\\
&=&\nu_j\Bigg(E^{-n}_j\bigg(\pi_j(\gamma_t^n\cap B_j(\epsilon) \big)\bigg)\cap\gamma \Bigg)>0.\end{eqnarray*}
The last inequality is obtained again because $\nu$ is fully supported. The definition of the reference measure \eqref{eq:refmeas2} completes the argument.

\subsubsection{The basins cover almost all.}

Finally, we proceed to prove that the set $\mathcal{B}(\mu_0)\cup\mathcal{B}(\mu_1)$ has a full $m$-measure. 

Let $Z=M\setminus (\mathcal{B}(\mu_0)\cup\mathcal{B}(\mu_1))$. Assume that $m(Z)>0$. Then, by the definition of $m$ \eqref{eq:refmeas2}, we can find that $\gamma_t=\gamma\times \{t\}$, where $\gamma$ is a closed disc in $\Td$ of radius less than $r_0$ and $t\in(0,1)$, such that 
$\nu_j(\pi_j(\gamma_t\cap Z))>0$. Arguing as in the proof of Lemma~\ref{le:baciasestodo}, taking a $\nu_j$-density point of $\pi_j(\gamma_t\cap Z)$, we can find a sequence of discs $\gamma_n\subseteq \gamma$ such that $E^n_j(\gamma_n)=:\sigma_n$. We may assume that the discs $\sigma_n$ have a radius bounded from below up to a subsequence and slightly reducing the size of $\sigma_n$ as $\sigma_n=\sigma_\infty$. The graph $\sigma^n_t:=F^n(\gamma_n\times\{t\})$ converges to the $C^1$-graph $\sigma^\infty_t$.

The magic occurs because $\pi_j$ is a semi-conjugation between $F$ and $E_j$. On the one hand, since $\theta$ is a density point with respect to $\nu_j$ of $\pi_j(\gamma_t\cap Z)$  and the disc $\gamma_n$ is small and close to $\theta$; thus, we obtain 
$$\nu_j\bigg(\pi_j\big(\sigma^\infty_t\cap [\mathcal{B}(\mu_0)\cup \mathcal{B}(\mu_1)] \big)\bigg)\approx0.$$
On the other hand, up to a straightforward adaptation of Lemma~\ref{le:curvasc1}, there exists a $\delta>0$ such that for any $n$ which is large enough, 
$$\nu_j\bigg(\pi_j\big(\sigma^n_t\cap [\mathcal B(\mu_0)\cup\mathcal B(\mu_1)]\big)\bigg)\geq\delta>0.$$
Therefore, we obtain a contradiction.

Finally, part (d) follows directly from the arguments used in the proof of Lemma~\ref{le:characterization2}. Assume $\mu\in M^{erg}(K)$ is such that $(\pi_j)_* \mu = \nu_j$ and $\mu\notin\{ \mu_0,\mu_1\}$. Then, given $\epsilon>0$ there exists a set $\Gamma_\epsilon\subseteq M$, such that for every $(\theta,t) \in\Gamma_\epsilon$ the $W^s(\theta,t)$ has length greater than $\epsilon$. Moreover, since $W^s(\theta,t)\subseteq \mathcal F^c(\theta, t)$, we have $\nu_j(\pi_j(\Gamma_\epsilon))=\mu(\Gamma_\epsilon)>0$. Therefore, the set 
$$\bigcup_{(\theta,t)\in\Gamma_\epsilon}W^s(\theta,t)\subseteq\mathcal B(\mu)$$ 
has a positive $m$-measure on $M$. It follows that $\mu=\mu_0$ or $\mu=\mu_1$ because the basins are disjoint and their union covers $m$-almost every point of the whole manifold $M$.

\endproof

Therefore, we have proven the Proposition~\ref{prop:medidasnegativasparaF}.


\subsection{Measure with non-negative Lyapunov exponents}\label{ssec:mmeiiF}

Now we present the highlight of the proof of the following statement.

\begin{prop}\label{prop:medidapositivaF}
Let $F:M\to M$ be  a $C^r$-local diffeomorphism, $r\geq 1$, satisfying \ref{F1}-\ref{F5}. 
Suppose that $\nu_j$ is any ergodic $E_j$-invariant probability measure with $\supp \nu_j=\Td$ and $J_{\nu_{j}} E_j$ has bounded distortion. Assume $\mu_j=\nu_j\times\delta_j$ satisfy the relation \eqref{eq:imagenmed}. Let $m$ be the reference measure defined by \eqref{eq:refmeas2} with $\nu=\nu_j$. 

Then:
\begin{enumerate}
        \item[(i)]       There are measurable, not continuous maps $\sigma_j: \Td \to M$, $\sigma_j(\theta)=(\zeta_j(\theta),\varsigma_j(\theta))$, such that for $\nu_j$-a.e. $\theta\in \Td$       
\begin{eqnarray*}
W^s((\theta, 0))&=&\{(\omega,t)\in\mathcal F^c(\theta,0)\::\:   t\in [0, \varsigma_0(\theta))\},\quad\mbox{ and }\\
W^s(\pi_{0,1}(\theta, 0))&=&\{(\omega,t)\in\mathcal F^c(\theta,0)\::\:   t\in (\varsigma_0(\theta), 1]\}.
\end{eqnarray*}       
We also have
\begin{eqnarray*}
W^s((\theta, 1))&=&\{(\omega,t)\in\mathcal F^c(\theta,1)\::\:   t\in (\varsigma_1(\theta), 1]\},\quad\mbox{ and }\\
W^s(\pi_{1,0}(\theta, 1))&=&\{(\omega,t)\in\mathcal F^c(\theta,1)\::\:   t\in [0, \varsigma_1(\theta))\}
\end{eqnarray*}       
Moreover,       
\begin{equation}\label{eq:uniqsepogen}
\sigma_j(\theta)=\sigma_{j^*}(\theta^*)\end{equation}     
where $\theta^*\in\Td$ is such that $\pi_{j,j*}(\theta,j)=(\theta^*,j^*)$. 

\item[(ii)] There exists a unique ergodic $F$-invariant probability measure $\mu$ such that $(\pi_j)_*\mu=\nu_j$ and $\mu\ne\mu_j$, $j=0,1$. Such measure $\mu$ is defined by
$$\int_M\Phi\,d\mu=\int_{\Td}\Phi(\sigma_j(\theta))\,d\nu_j(\theta),$$
for every $\Phi:M\to \mathbb{R}$ continuous.
        \item[(iii)] Moreover, $\mu$ has a nonnegative central Lyapunov exponent, is approximated by measures supported along periodic orbits in the interior of $M$. If  the leaves of $\mathcal F^c$ are $C^2$, then  $\supp\mu=M$.      
    \end{enumerate}
\end{prop}

\proof Since $\mu_j=\nu_j\times\delta_j$ have central Lyapunov exponent $\lambda^c(\mu_j)<0$, we can define separating map as before. Denote by $\sigma_j(\theta)=(\zeta_j(\theta),\varsigma_j(\theta))$  the boundary of the connected component of the Pesin stable manifold of $(\theta,j)$. As before, such stable manifold is a curve inside $\mathcal F^c(\theta,j)$ with extremal point $(\zeta_j(\theta),\varsigma_j(\theta))\in \mathcal F^c(\theta,j)$.  Hence,
\begin{eqnarray*}
W^s((\theta, 0))&=&\{(\omega,t)\in\mathcal F^c(\theta,0)\::\:   t\in [0, \varsigma_0(\theta))\},\quad\mbox{ and }\\
W^s((\theta,1))&=&\{(\omega,t)\in\mathcal F^c(\theta,1)\::\:   t\in (\varsigma_1(\theta), 1]\}.
\end{eqnarray*}
Therefore, we can define the separating map $\sigma_j: \Td  \to M$ as above. Note that $\varsigma_j(\theta)\ne j$ for $\nu_j$-almost every $\theta\in \Td$. Moreover,  we can prove as in Lemma~\ref{le:uniquenessseparating} that for $\nu_j$-almost every point $\theta\in\Td$, 
$$\sigma_j(\theta)=\sigma_{j^*}(\theta^*),$$ 
where $\theta^*\in\Td$ is such that $\pi_{j,j*}(\theta,j)=(\theta^*,j^*)$.

Note that due to \ref{F3ii}, the map $\zeta_j:\Td\to \Td$ is continuous. Nevertheless, the map $\varsigma_j:\Td\to I$ is measurable but is not continuous. Otherwise, the continuous curve $\{\sigma_j(\theta) \::\: \theta\in\Td\}$  separates the intermingled basins $\mathcal{B}(\mu_0)$ and $\mathcal{B}(\mu_1)$.
Therefore, we have item (i). 

To prove item (ii), we proceed as above by invoking the Rokhlin's disintegration theorem: there exists a family of conditional probabilities $\{\mu_{(\theta,j)}\::\:\theta\in\Td\}$ with respect to the measurable partition $\{\mathcal F^c(\theta,j)\::\:\theta\in\Td\}$, $j=0,1$, such that
\begin{equation}\label{eq:defmuF}
\int_M\Phi\,d\mu=\int_{\Td}\int_{\mathcal F^c(\theta,j)}\Phi\,d\mu_{(\theta,j)}\,d\nu_j(\theta),
\end{equation}
\noindent for every $\Phi:M\to \mathbb{R}$ continuous. Note that in this case the quotient measure $\hat\mu$ coincides with $\nu_j$. It follows from the definition of the map $\sigma_j$ that, if $\mu\ne\mu_j$, then for $\nu_j$-almost every $\theta\in \Td$, $\mu_{(\theta,j)}=\delta_{\sigma_j(\theta)}$ and $\supp\mu_{(\theta,j)}=\{\sigma_j(\theta)\}$. By the  uniqueness of the Rokhlin decomposition and \eqref{eq:uniqsepogen}, we have that the measure defined by \eqref{eq:defmuF} is the unique ergodic $F$-invariant probability  such that $\mu\ne\mu_j$, $j=0,1$ and $(\pi_j)^*\mu=\nu_j$.

Proposition~\ref{prop:medidasnegativasparaF} item (d) implies that $\lambda^c(\mu)\geq 0$. 

The approximation in the weak* topology of $\mu$ by probability measures supported on periodic points in the interior of $M$ follows from the approximation of $\nu_j$ by probability measures supported on periodic points on $\Td\times\{j\}$. With fixed $j=0,1$, if $\theta\in \Td$ is  a periodic point for $E_j$, then $\theta^*=\pi_{j^*}(\theta,j)$ is a periodic point for $E_{j^*}$ having the same period and the measures
$$\delta\theta:=\frac1n\sum_{k=0}^{n-1} \delta_{E^k_j(\theta)},\quad\mbox{and}\quad \delta\theta^*:=\frac1n\sum_{k=0}^{n-1} \delta_{E^k_{j^*}(\theta^*)}$$
satisfy $(\pi_{j^*})^*\delta\theta=\delta\theta^*$. Therefore,  we can proceed as in  Lemma~\ref{le:existence} and obtain a sequence $\theta_n\in\Td$ of periodic points on $\Td$ such that $\delta\theta_n$ converges to $\nu_j$ in the weak* topology. Moreover, up to a subsequence, $\lambda^c(\delta\theta_n\times\delta_j)<0$ for both $j=0$ and $j=1$. Again we can find a periodic orbit $(\omega_n,t_n)\in\mathcal F^c(\theta_n,j)$ with $(\omega_n,t_n)=\sigma_j(\theta_n)$. Finally, consider $\tilde\mu$ any ergodic component of any accumulation point of the periodic measures 
$$ \mu_n=\frac1n\sum_{k=0}^{n-1} \delta_{K^k(\sigma_j(\theta_n))}.$$
Then $\lambda^c(\tilde\mu)\geq 0$ and $\pi_*\tilde\mu=\nu$, and by the uniqueness of measure $\tilde\mu=\mu$.

Finally, the proof of $\supp\mu=\Td$ is straightforward if we assume that $F$ and the leaves of $\mathcal{F}^c$ are $C^2$, and therefore we complete the proof of Theorem~\ref{mainteo:F}.
We point out that in this general case we do not discard the case $\lambda^c(\mu)=0$. Recall that from Subsection~\ref{ssec:zerocenterexponent}, to discard the null central Lyapunov exponent we invoke the invariance principle (Lemma~\ref{le:klemma} and Theorem~\ref{teo:princinv}) to contradict the intermingledness of the basins. Applying the invariance principle in this general case is harder because neither $ F $ nor its extension to the inverse limit have the structure of the skew product.


\section{Applications}\label{secc:applications}

\subsection{Perturbations of Kan-like maps}\label{ssec:perturbations}

Let us denote by ${\rm End}^r(M,\partial M)$ the set of local diffeomorphisms preserving the boundary provided with the $C^r$-topology, $r\geq 1$. 

\begin{prop}\label{prop:condabiertas}
Let $K:M\to M$ be a Kan-like map defined by \ref{K1}-\ref{K3}. Then, there exists a $C^r$-neighborhood $\mathcal U\subseteq {\rm End}^r(M,\partial M)$, with $K\in \mathcal U$ such that every $F\in\mathcal U$ satisfies \ref{F1}-\ref{F4}. 

\end{prop}

\proof Let us quickly check why each one of the following statements are verified: 

\begin{enumerate}
\item $F$ preserves the boundary.
\item $F$ is a partially hyperbolic endomorphism.   
\item $F$ preserves an invariant center foliation satisfying \ref{F3}.
\item $F$ relates the dynamics along the boundary through a heteroclinical cycle.
\end{enumerate}

Let $F\in {\rm End}^r(M,\partial M)$ be defined by
$$F(\theta,t)=(E(\theta,t),\psi(\theta,t)).$$
$F$ preserves the boundaries by definition, so for every $\theta\in \Td$, $\psi(\theta,j)=j$, $j=0,1$.

The second item follows because partial hyperbolicity is a $C^r$-open property, $r\geq 1$ (see Subsection~\ref{ssec:hiperbparcial}). Therefore, the restriction of $F$ to each boundary $\Td\times\{j\}$ defines an expanding map on the torus $E_j\::\:\Td\to\Td$. Moreover,  in the set of endomorphisms on the torus ${\rm End}^r(\Td)$, each $E_j$ is $C^r$ close from $E$. 

Recall that the Kan-like map $K$ has a $C^r$-foliation, $r\geq 1$, in compact segments $\mathcal{F}^c_K=\{I_\theta=\{\theta\}\times I\::\:\theta\in\Td\}$ which is normally hyperbolic and plaque expansive. It follows from the classical theory of invariant foliations (see Theorem~\ref{teo:foliaciones}) that we can reduce the neighborhood $\mathcal U$ so that, in addition to the properties above, every $F\in\mathcal U$ has a unique $F$-invariant foliation $\mathcal F^c$ tangent to the center direction $E^c$, which is normally hyperbolic and plaque expansive. Moreover, there are bi-H\"older homeomorphisms $h:M\to M$ such that $(K,\mathcal F^c_K)$ and $(F,\mathcal F^c_F)$ are leaf conjugated. Moreover, the holonomies are H\"older continuous. 

Finally, to verify item (4) we note that when reducing $\mathcal U$, if it is necessary, each $E_j$, $j\in\{0,1\}$, has two fixed points $p_j$, $q_j$ which are the hyperbolic continuation of $p$ and $q$, the fixed points of $E$. Moreover, $\pi_{j,j^*}(p_j)=p_{j^*}$ and $\pi_{j,j^*}(q_j)=q_{j^*}$, and so the fibers
$$\mathcal F^c(p_0,0)=\mathcal F^c(p_1,1)\quad \mbox{and}\quad \mathcal F^c(q_0,0)=\mathcal F^c(q_1,1)$$ are fixed by $F$. Reducing $\mathcal U$ if it is necessary, the dynamics of $F$ along the fixed leaf $\mathcal F^c(p_0,0)$  have exactly two fixed points: a hyperbolic source at $(p_1,1)$, and a hyperbolic sink in $(p_0,0)$. Analogously, the dynamics of $F$ along the fixed leaf $\mathcal F^c(q_0,0)$ have exactly two fixed points: a hyperbolic sink at $(p_1,1)$, and a hyperbolic source in $(p_0,0)$. Therefore, $F$ relates the dynamics along the boundary through a heteroclinical cycle.
\endproof


\subsection{Equilibrium states for Kan-like maps}\label{ssec:mmecaracterization}
In this section, we address the existence of equilibrium states of $K$. We consider H\"older continuous potentials $\Phi:M\to \mathbb R$ which are constant along the fiber: 
$$\Phi(\theta,t)=\Phi(\theta,0),\quad\mbox{ for every } t\in I.$$
This class of potentials includes the particular case $\Phi=0$, which is interesting itself because the equilibrium states are the measures maximizing the topological entropy. Let us denote by $M_\Phi^{erg}(K)$ the set of ergodic equilibrium states for $\Phi$.

If $\phi\::\Td\to\mathbb R$ is defined as 
\begin{equation}\label{eq:phiPhi}
\phi(\theta)=\Phi(\theta,0),\quad\theta\in\Td,
\end{equation}
then $\phi$ is a H\"older continuous potential and by Ruelle's Theorem, $E$ has a unique equilibrium state $\nu_\phi$. Moreover, $\nu_\phi$ has a H\"older continuous Jacobian $J_{\nu_\phi}E$ (see Theorem~\ref{teo:Ruelle}).

\begin{prop}\label{prop:Bgen} 
Assume $K$ is a Kan-like map and let $\Phi\::\: M\to \mathbb R$ be H\"older continuous and constant along the fibers. Let $\nu_\phi$ be the unique equilibrium state of $E$ for the  potential $\phi(\theta)=\Phi(\theta,0)$, $\theta\in\Td$. Assume that
\begin{equation*}\label{eq:condK4.2mme}
\int\log |\partial_t \varphi (\theta,j)|\: d\nu_\phi(\theta)<0, \quad j=0,1;
\end{equation*}
\noindent Then, there exist exactly three equilibrium states $\mu_0, \mu$, and $\mu_1$ of $K$ for the potential $\Phi$. All of them are hyperbolic and 

\begin{enumerate}
\item $\mu_0$ and $\mu_1$ have negative central Lyapunov exponents and are supported on the boundaries, and each basin has positive weight with respect to the reference measure $m_\Phi=\nu_\phi\times \m_I$ and the union of their basin covers $m_\Phi$-a.e. the whole manifold $M$. Moreover, their basins are intermingled with respect to $m_\Phi$. 
\item $\mu$ has a positive central Lyapunov exponent and is approximated by ergodic measures supported on periodic orbits of $K$ in the interior of $M$. Moreover, if $K$ is $C^2$ then $\mu$  is fully supported on $M$.
\end{enumerate}
\end{prop}

It is an immediate consequence of Theorem~\ref{mainteo:KE1} that the set
$$\{\eta\in M^{erg}(K)\::\: \pi_*\eta=\nu_{\phi}\}=\{\mu_0,\mu,\mu_1\},$$
where $\mu_j=\nu_\phi\times\delta_j$ satisfies statement (1) and $\mu$ provided by Theorem~\ref{mainteo:KE1} part~\ref{TeoAii} satisfies statement (2).

It remains to be proven that 
$$M_\Phi^{erg}(K)=\{\eta\in M^{erg}(K)\::\: \pi_*\eta=\nu_{\phi}\},$$
and we focus our efforts on this. For this purpose, we prove that in the case of Kan-like maps, the topological pressures of $K$ and $E$ with the potentials defined above coincide. 
 
 \begin{lema}\label{le:entropyzero} For every $\theta\in\Td$, we have
$$\htop(K,\pi^{-1}(\theta))=0.$$
\end{lema}
\proof For fixed $\theta\in\Td$, we have $\pi^{-1}(\theta)=\{\theta\}\times[0,1]$ with bounded length $L=1$. Then, for every $\epsilon>0$ and $n\geq 1$, every $(n,\epsilon)$-separated subset $\mathcal{E}\subseteq \pi^{-1}(\theta)$ has cardinality
$$|\mathcal{E}|\leq n\left( \frac1\epsilon +1\right).$$ 
Then,
$$\frac{1}{n}\log s(n,\epsilon,\pi^{-1}(\theta))\leq \frac1n\left(\log n+ \log \left( \frac1\epsilon +1 \right) \right)$$
and taking $n\to\infty$, we obtain the result.
\endproof

\begin{lema} \label{le:entropy}  $\Ptop( K,\Phi)= \Ptop (E,\phi)$. 

\end{lema}

\proof Since $E=K|(\Td\times\{0\})$ (see \cite[Theorem 9.8]{Wal}), 
\begin{equation}\label{eq:2.1.1}
\Ptop(E,\phi)\leq \Ptop(K,\Phi).
\end{equation}
On the other hand, since $K$ and $E$ are semiconjugated, it follows from \eqref{eq:CarvalhoPerez} that
\begin{equation}\label{eq:2.1.2}
\Ptop(K, \Phi) \leq \Ptop (E,\phi ) + \sup \left\{ \int \htop(K, \pi^{-1}(y)) d\nu :\nu\in M^1(E) \right\}
\end{equation}
The conclusion follows directly from \eqref{eq:2.1.1}, \eqref{eq:2.1.2} and Lemma~\ref{le:entropyzero}.
\endproof

Finally, we use the Ledrappier-Walters result to establish that the set of measures maximizing the pressure of $K$ coincides with the subset of ergodic measures that projects down to the equilibrium state for $E$. 

\begin{lema}\label{le:characterization} $M_\Phi^{erg}(K)=\{\eta\in M^{erg}(K)\::\: \pi_*\eta=\nu_{\phi}\}$.
\end{lema}

\proof From the Ledrappier-Walters formula \eqref{eq:LW} and Lemma~\ref{le:entropyzero}, we have
\begin{eqnarray*}
\sup_{\eta\in M^1(K)} \{{\rm h}_\eta(K)\::\: \pi_*\eta=\nu_\phi\}&=& {\rm h}_{\nu_\phi}(E)+\int_{\Td}\htop(K,\pi^{-1}(\theta))d\nu_\phi(\theta)\\
&=&{\rm h}_{\nu_\phi}(E).
\end{eqnarray*}
If $\eta\in M^{erg}(K)$ is such that $\pi_*\eta=\nu_\phi$, then $\int \Phi \ d\eta= \int \phi \ d\nu_\phi$ and so
\begin{eqnarray*}
\sup_{\eta\in M^1(K)} \{{\rm h}_\eta(K) + \int \Phi d\eta \::\:  \pi_*\eta=\nu_\phi\}&=&{\rm h}_{\nu_\phi}(E)+\int \phi \,d \nu_\phi\\ &=&P_{\rm{top}}(E,\phi)=P_{\rm{top}}(K,\Phi),\end{eqnarray*}
\noindent where the last equality comes from Lemma~\ref{le:entropy}. Hence, $\eta$ is an equilibrium state for $\Phi$ and so $\eta\in M^{erg}_\Phi(K)$.

Reciprocally, if $\eta_\Phi\in M_\Phi^{erg}(K)$, then $\pi_*{\eta_\Phi}=:\nu$ is an $E$-invariant measure. Again by the Ledrappier-Walters formula
$${\rm h}_{\eta_\Phi}(K)\leq {\rm h}_{\nu}(E)+\int_{\Td}\htop(K,\pi^{-1}(\theta))d\nu(\theta),$$
and since $\htop(K,\pi^{-1}(\theta))=0$ for every $\theta\in\Td$ (Lemma~\ref{le:entropyzero}) and $\int \Phi \ d\eta_\Phi= \int \phi \ d\nu$ when $\phi$ is defined by \eqref{eq:phiPhi}, we have
$$P_{\rm{top}}(K,\Phi)={\rm h}_{\eta_\Phi}(K)+\int \Phi \ d\eta_\Phi\leq {\rm h}_{\nu}(E)+\int \phi \ d\nu.$$
Since $\eta_\Phi$ is an equilibrium state for $\Phi$, from Lemma~\ref{le:entropy} we have
$$P_{\rm{top}}(E,\phi)=P_{\rm{top}}(K,\Phi)\leq {\rm h}_{\nu}(E)+\int \phi \ d\nu\leq P_{\rm{top}}(E,\phi).$$
Therefore, $\nu$ is an equilibrium state for $(E,\phi)$, and by uniqueness $\nu=\nu_\phi$.
\endproof

This completes the proof of Proposition~\ref{prop:Bgen} taking $\Phi=0$.


\subsection{Negative center Lyapunov exponent condition}\label{ssec:condnegLyap}

Let $K\::\:M\to M$ be a Kan-like map defined by
$$K(\theta,t)=(E(\theta),\varphi(\theta,t)).$$
Let  $\nu$ be an ergodic invariant probability measure for $E$ with $\supp\nu=\Td$.

Along this section we discuss about conditions to guarantee that the hypothesis \eqref{eq:condK4im}:
\begin{equation}\label{eq:condLEneg}
\int_{\Td}\log|\partial_t\varphi(\theta,j)|\,d\,\nu(\theta)<0,\quad j=0,1,
\end{equation}
is satisfied. 

Having in mind that $K$ preserves the boundary we are looking for $C^r$ functions $\varphi:M\to I$ such that
$$\varphi(\theta,j)=j,\quad j=0,1.$$
Moreover, since the center foliation should be normally hyperbolic, then we can assume that $\varphi$ is $C^r$-close of the identity on $t$, $r\geq 1$.  We can therefore suppose that  $\varphi\::\: M\to I$ defined by
$$\varphi (\theta,t)=t+\epsilon\psi(\theta,t)$$
where $\psi\::\:M\to I$ is a map $C^r$ close to 0, such that  $\psi(\theta,j)=0$, for every $\theta\in \Td$ and $j=0,1$.  We want to impose conditions on $\psi$ in order to conclude \eqref{eq:condLEneg}.

Since $\log(1+x)=x-x^2/2+O(x^3)$ then for every $\theta\in \Td$, we have
\begin{eqnarray*}
\log|\partial_t\varphi(\theta,j)|&=&\log\big(1+\epsilon\cdot\partial_t\psi(\theta,j)\big)\\
&=&\epsilon\cdot\partial_t\psi(\theta,j)-\frac{\epsilon^2}{2}\left[\partial_t\psi(\theta,j)\right]^2+O\left(\left[\epsilon\cdot\partial_t\psi(\theta,j)\right]^3\right).
\end{eqnarray*}
Denote by $R=\sup_{\theta\in\Td}|\partial_t\psi(\theta,j)|$. If we assume that
\begin{equation}\label{eq:condneg1}
\int_{\Td}\partial_t\psi(\theta,j)d\nu(\theta)=0,\quad \mbox{ but }\quad \partial_t\psi(\theta,j)\not\equiv 0,
\end{equation}
 then we can choose $\epsilon>0$ and $R>0$ so that
$$\int_{\Td}\log|\partial_t\varphi(\theta,j)|d\nu(\theta)=-\frac{\epsilon^2}{2}\int_{\Td}\left[\partial_t\psi(\theta,j)\right]^2d\nu(\theta)+\epsilon^3O(R^3)<0.$$
As a particular case, we can consider $\psi(\theta,t)=C(\theta)\xi(t)$, where $C$ and $\xi$ are $C^r$, and $\xi(j)=0$, $j=0,1$. Therefore, condition \eqref{eq:condneg1} is satisfied if
\begin{equation}\label{eq:condneg2}
\int_{\Td}C(\theta)\,d\nu(\theta)=0,\quad \mbox{ but }\quad C(\theta)\xi'(j)\not\equiv 0.
\end{equation}
Note that  $C(\theta)=\cos(2\pi\theta)$ and $\xi(t)=t(1-t)$ satisfy \eqref{eq:condneg2} when $\nu=\ms$ and so, the original Kan example satisfies \eqref{eq:condLEneg}. Of course, it is not difficult to exhibit examples of continuous functions (not identically zero) $C:\Td\to\mathbb R$, $\theta\in\Td$, such that
$$\int C(\theta)d\md(\theta)=0.$$
Summarizing we have
\begin{prop}\label{prop:condNLE1}
Let $\psi\::\: M\to I$ be $C^r$, $r\geq 1$, and $\psi(\theta,j)=0$, for $j=0,1$. Assume that
$$\int_{\Td}\partial_t\psi(\theta,j)d\nu(\theta)=0,\quad \mbox{ but }\quad \partial_t\psi(\theta,j)\not\equiv 0,$$
then, there exist $\epsilon>0$ such that $\|\psi\|_{C^1}<\epsilon$, then  \eqref{eq:condLEneg} is satisfied. 
\noindent In particular, if $\psi(\theta,t)=C(\theta)\xi(t),$ where $C$ and $\xi$ are $C^r$, $r\geq 1$, $\xi(j)=0$,  and
$$\int_{\Td}C(\theta)\,d\nu(\theta)=0,\quad \mbox{ but }\quad C(\theta)\xi'(j)\not\equiv 0,$$
then the same conclusion above holds.
\end{prop}

Now we are going to discuss the of the robustness of condition \eqref{eq:condLEneg}. Assume that $K\::\:M\to M$ is a Kan-like map defined by
$$K(\theta,t)=(E(\theta),\varphi(\theta,t)),$$
and $\nu$ be an ergodic invariant probability measure for $E$ with $\supp\nu=\Td$.

Suppose that $\varphi$  satisfies \eqref{eq:condLEneg} (for instance if $\psi$ satisfies some of  the conditions in Proposition~\ref{prop:condNLE1}). It is consequence of the continuity of the functional 
$$M^1(E)\ni\eta\to\int_{\Td}\log|\partial_t\varphi(\theta,j)|d\eta(\theta)\in\mathbb R$$
that if $\eta$ is close to $\nu$ in the weak* topology, then $\eta$ also satisfies \eqref{eq:condLEneg}. 

In the particular case where $E$ is an affine expanding map and $\nu=\md$ the condition  \eqref{eq:condLEneg} is satisfied for every $\eta$ full supported close to Lebesgue. This is precisely the case for certain equilibrium states associated  as we discuss below. 

If $E$ is an  affine expanding map, then  its maximum entropy measure coincides with its physical measure, which is Lebesgue. Equivalently,  $\md$ is the  equilibrium state associated to the null potential $\phi_0$ and geometric potential $\phi_{\mbox{geo}}$, respectively.
It follows from statistical stability (Theorem \ref{teo:StabstatVV}) we have that for every continuous H\"older potential $\phi$ close to $\phi_0$ (or $\phi_{\mbox{geo}}$), then $\phi$ satisfies \eqref{eq:boundvar}  and  its equilibrium state $\nu_{\phi}$ is close to $\md$  in the weak* topology. Of course by the remark above $\nu_\phi$ satisfies \eqref{eq:condLEneg}. Summarizing,

\begin{prop}\label{prop:condNLE2}
Assume that $K\::\:M\to M$ is a $C^r$, $r\geq 1$, Kan-like map defined by
$$K(\theta,t)=(E(\theta),\varphi(\theta,t)),$$
and $\nu$ is an ergodic invariant probability measure for $E$ with $\supp\nu=\Td$. Then, there exist an open neighborhood $\mathbb U\subseteq M^1(E)$, $\nu\in\mathbb U$, such that every $\eta\in \mathbb U$ satisfies \eqref{eq:condLEneg}.

If $K$ is an affine Kan like map, then for every continuous H\"older potential $\phi: \Td\to \mathbb R$ close to $\phi_0$   its equilibrium state $\nu_{\phi}$  satisfies \eqref{eq:condLEneg}.
\end{prop}

Also we can consider perturbation of the Kan-like maps inside the class of skew product in the cylinder preserving the boundary. More precisely, we consider  $K^r(M,\partial M)$  as the set of  $C^r$ Kan-like map defined by
$$K(\theta,t)=(E(\theta),\varphi(\theta,t)).$$
Fix $K_0\in K^r(M,\partial M)$ defined by
$$K_0(\theta,t)=(E_0(\theta),\varphi_0(\theta,t)),$$
and  consider $\nu_0$  an ergodic invariant probability measure for $E$ with $\supp\nu_0=\Td$ and satisfying \eqref{eq:condLEneg}. The bilinear functional $\Phi:K^r(M,\partial M)\times \mathbb M^1(M)\to\mathbb R$ defined by
$$\Phi(E,\nu)=\int_{\Td}\log|\partial_t\varphi(\theta,j)|d\nu(\theta)$$
is continuous. Hence, there exist  neighborhoods $\mathcal U\subseteq K^r(M,\partial M)$, $K_0\in \mathcal U$ and $\mathbb U\subseteq \mathbb M^1(M)$, $\nu_0\in \mathbb U$, such that for every $K\in \mathcal U$ and  $\nu\in M^1(E)\cap\mathbb U$, we have
$$\int_{\Td}\log|\partial_t\varphi(\theta,j)|d\nu(\theta)<0.$$
Therefore, in the particular case where $K_0$ is an affine Kan-like map and  $\nu_0=\md$ is the equilibrium state for the potential $\phi_0^{E_0}$  and $\phi_{\mbox{geo}}^{E_0}$,  we obtain the neighborhood $\mathcal U$ and $\mathbb{U}$ as above. We can invoke again the statistical stability (Theorem \ref{teo:StabstatVV}) and we have that for every $K\in \mathcal U$, every continuous H\"older potential $\phi^E$ close to $\phi_0^{E_0}$ (or $\phi_{\mbox{geo}}^{E_0}$), then $\phi^E$ satisfies \eqref{eq:boundvar}  and  its equilibrium state $\nu_{\phi}^E\in\mathbb U$. Of course by the remark above $\nu_\phi^E$ satisfies \eqref{eq:condLEneg}. In particular this scheme allows to fix the potential ($\phi_0$ for instance) for every dynamic $K$. Summarizing,

\begin{prop}\label{prop:condNLE3}
Fix $K_0\in K^r(M,\partial M)$ defined by
$$K_0(\theta,t)=(E_0(\theta),\varphi_0(\theta,t)),$$
and  consider $\nu_0$  an ergodic invariant probability measure for $E_0$ with $\supp\nu_0=\Td$ and satisfying \eqref{eq:condLEneg}. 
Then, there exist  neighborhoods $\mathcal U\subseteq K^r(M,\partial M)$, $K_0\in \mathcal U$ and $\mathbb U\subseteq \mathbb M^1(M)$, $\nu_0\in \mathbb U$, such that for every $K\in \mathcal U$, any  $\nu\in M^1(E)\cap\mathbb U$ satisfies \eqref{eq:condLEneg}.

In particular, if $K_0$ is an affine Kan-like map  such that $\md$ satisfies \eqref{eq:condLEneg}, there exist  neighborhoods $\mathcal U\subseteq K^r(M,\partial M)$, $K_0\in \mathcal U$ and $\mathbb U\subseteq \mathbb M^1(M)$, $\md\in \mathbb U$, such that for every $K\in \mathcal U$, and  for every continuous H\"older potential $\phi: \Td\to \mathbb R$ close to $\phi_0$, the equilibrium state $\nu_{\phi}^E$  corresponding to the pair $(E,\phi)$ satisfies \eqref{eq:condLEneg}.
\end{prop}

Finally, we can consider general perturbation of Kan-like maps inside the class ${\rm End}^r(M,\partial M)$ of local diffeomorphims in the cylinder preserving the boundary. Fix $K\in {\rm End}^r(M,\partial M)$ defined by
$$K(\theta,t)=(E(\theta),\varphi(\theta,t)),$$
and as above, consider $\nu$  an ergodic invariant probability measure for $E$ with $\supp\nu=\Td$ and satisfying 
\begin{equation}\label{eq:condLEneg2}
\int_{\Td}\log\|DK|E^c(\theta,j)\|\,d\,\nu(\theta)<0,\quad j=0,1,
\end{equation}
The map ${\rm End}^r(M,\partial M)\times M\to\mathbb  R$, defined by $(F,\theta,t)\to \|DF|E^c(\theta,t)\|$ is continuos, as well as the bilinear functional $\Phi:{\rm End}^r(M,\partial M)\times\mathbb M^1(M)\to\mathbb R$ defined by
$$\Phi(F,\nu)=\int_{\Td}\log\|DF|E^c(\theta,j)\|d\nu(\theta).$$
Hence, there exist  neighborhoods $\mathcal U\subseteq {\rm End}^r(M,\partial M)$, $K\in \mathcal U$ and $\mathbb U\subseteq M^1(M)$, $\nu_0\in \mathbb U$, such that for every $F\in \mathcal U$ and  $\nu\in M^1(E_j)\cap\mathbb U$, we have
\begin{equation}\label{eq:condLEneg3}
\int_{\Td}\log\|DF|E^c(\theta,j)\|d\nu(\theta)<0.
\end{equation}
Therefore, in the particular case where $K$ is an affine Kan-like map and  $\nu_0=\md$. Note that in this case
$$DK|E^c(\theta,j)=\partial_t\varphi(\theta,t),$$
and so,  if $\md$ satisfies condition \eqref{eq:condLEneg}, then we can obtain the neighborhood $\mathcal U$ and $\mathbb{U}$ above. Invoking once again the statistical stability (Theorem \ref{teo:StabstatVV}) we have that, for every $F\in \mathcal U$, every continuous H\"older potential $\phi^{E_j}$ close to $\phi_0^{E}$ (or $\phi_{\mbox{geo}}^{E}$), then $\phi^{E_j}$ satisfies \eqref{eq:boundvar}  and  its equilibrium state (with respect to $E_j$ $j=0,1$) $\nu_{\phi}^{E_j}\in\mathbb U$. Of course by the remark above $\nu_\phi^{E_j}$ satisfies \eqref{eq:condLEneg2}. In particular this scheme allows to fix the potential (for instance $\phi_0$) for every dynamic $F$ close to $K$. Summarizing,

\begin{prop}\label{prop:condNLE4}
Suppose that $K\in {\rm End}^r(M,\partial M)$ is a Kan-like map and let $\nu$   be an ergodic invariant probability measure for $E$ with $\supp\nu=\Td$ and satisfying \eqref{eq:condLEneg2}. 
Then, there exist  neighborhoods $\mathcal U\subseteq {\rm End}^r(M,\partial M)$, $K\in \mathcal U$, and $\mathbb U\subseteq M^1(M)$, $\nu\in \mathbb U$, such that for every $F\in \mathcal U$, any  $\nu\in M^1(E_j)\cap\mathbb U$ satisfies \eqref{eq:condLEneg3}.

In particular, if $K_0$ is an affine Kan-like map such that $\md$ satisfies \eqref{eq:condLEneg} , then there exist  neighborhoods $\mathcal U\subseteq  {\rm End}^r(M,\partial M)$, $K_0\in \mathcal U$, and $\mathbb U\subseteq \mathbb M^1(M)$, $\md\in \mathbb U$, such that for every $F\in \mathcal U$, and  for every continuous H\"older potential $\phi: \Td\to \mathbb R$ close to $\phi_0$, the equilibrium state $\nu_{\phi}^{E_j}$  corresponding to the pair $(E_j,\phi)$, $j=0,1$, satisfies \eqref{eq:condLEneg3}.

\end{prop}

%
%
%
%
%


\subsection{Measures of maximal entropy}\label{ssec:mmecaracterization}

Finally, we prove Theorem~\ref{mainteo:C}. 

Let $K:M\to M$ be a Kan-like map defined by \ref{K1}-\ref{K3}. Let $\nu_0$ be the unique measure of maximal entropy of $E$. We assume that
$$
\int\log |\partial_t \varphi (\theta,j)|\: d\nu_0(\theta)<0, \quad j=0,1.
$$
From Proposition~\ref{prop:condabiertas}, there exists a $C^r$-neighborhood $\mathcal U\subseteq {\rm End}^r(M,\partial M)$, $K\in \mathcal U$, such that every $F\in\mathcal U$, defined by
$$F(\theta,t)=(E^F(\theta,t), \varphi^F(\theta,t)),\quad(\theta,t)\in M,$$
satisfies \ref{F1}-\ref{F4}. 

We denote by $E^F_j:\Td\to\Td$ the expanding map defined by $E^F_j(\theta)=E^F(\theta, j)$, $j=0,1$. Recall that we have a natural identification among the (ergodic) $E_j^F$-invariant probability measures and the $F$-invariant probability measures supported on the boundary $\Td\times\{j\}$. We denote by $\nu_j^F$ the unique probability measure maximizing the entropy of $E_j^F$. 

Note that if $F$ is $C^r$ close to $K$, then the uniformly expanding map $E^F_j$ is $C^r$ close to the expanding maps $E$, and then $\nu_j^F\to\nu_0$ in the weak*
topology as $F\to K$ (see Lemma~\ref{le:statest}). 

On the other hand, the center direction $E^c_F$ of $F$ varies $C^0$ with respect to $F$, and so we have
$$\lambda^c(F,\mu^F_j)=\int\log\|DF|E^c_F\|d\nu^F_j\to \int\log|\partial_t\varphi(\theta,j)|d\nu_0=\lambda^c(K,\nu_0)<0,$$
where $\mu^F_j=\nu^F_j\times \delta_j$.

Therefore, reducing $\mathcal U$ if it is necessary, we determine that every $F\in \mathcal U$ satisfies the conditions of Proposition~\ref{prop:medidasnegativasparaF} and Proposition~\ref{prop:medidapositivaF}. Recalling that $\nu^F_j$ is an equilibrium state for the expanding map $E^F_j$, and hence $\supp \nu_j^F=\Td$ and the state has a Jacobian $J_{\nu_j^F}E^F_j$ with bounded distortion,  from Proposition~\ref{prop:medidasnegativasparaF} we have that
\begin{itemize}
        \item[(a)] the measures $\mu_j^F$, $j=0,1$, are observable with respect to $m^F$,
        \item[(b)] the basins of $\mu_j$ are intermingled with respect to $m^F$,
        \item[(c)] the union of their basins $\mathcal{B}(\mu_0^F)\cup \mathcal{B}(\mu_1^F)$  covers $m^F$-a.e. the whole manifold $M$, 
\end{itemize}
where $m^F$ is the Borel probability measure defined by
$$ m^F(A)=\int (\pi_{t,j})_*\nu^F_j(A)\,d\,\m_{I}(t); \mbox{ for every }A\subseteq M \mbox{ measurable.} $$
\noindent Again from Proposition~\ref{prop:medidasnegativasparaF}, we have that
        
\begin{itemize}
        \item[(d)] there is no other invariant measure $\mu$ such that $\pi_*\mu=\nu^F_j$ with a negative center Lyapunov exponent. 
\end{itemize}

On the other hand, from Proposition~\ref{prop:medidapositivaF} we have that

\begin{itemize}
        \item[(e)] There is a measurable, not continuous, map $\sigma^F: \Td \to I$ such that for $\nu^F_j$-a.e. $\theta\in \Td$
\begin{eqnarray*}
W^s_F((\theta, 0))&=&\{(\zeta,t)\in\mathcal F^c(F;(\theta,0))\::\:   t\in [0, \sigma(\theta))\},\quad\mbox{ and }\\
W^s_F(\pi_1(\theta, 0))&=&\{(\zeta,t)\in\mathcal F^c(F;(\theta,0))\::\:   t\in (\sigma(\theta), 1]\}.
\end{eqnarray*}
       \item[(f)] There exists a unique ergodic $F$-invariant probability measure $\mu^F$ such that $\pi_j^*\mu^F=\nu^F_j$ and $\mu^F\ne\mu_j^F$, $j=0,1$. Such measure $\mu$ is defined by
$$\int_M\Phi\,d\mu^F=\int_{\Td}\Phi(\theta,\sigma(\theta))\,d\nu^F_j(\theta),$$
for every continuous $\Phi:M\to \mathbb{R}$.
        \item[(g)] Moreover, $\mu^F$ has a nonnegative central Lyapunov exponent, is approximated by measures supported along periodic orbits in the interior of $M$ and $\supp\mu^F=M$ when $K$ is $2$-partially hyperbolic (see Lemma~\ref{le:transitive}).      
\end{itemize}

Let $\{F_n\}_{n\geq 1}$ be a sequence converging to $K$ in the $C^r$-topology. Let $\eta$ be an accumulation measure of $\mu_n$, where $\mu_n$ is the measure obtained from (f) above for $F_n$. Then, $\nu_n=\pi_j^n\circ\mu_n\to\pi_j\eta=\nu$ by Lemma~\ref{le:statest}. Hence, $\eta=\mu^K$, where $\mu^K$ is the ``central measure'' obtained from Proposition~\ref{prop:Bgen} part (2) for the measure of the maximal entropy of $K$. It follows from the convergence above that reducing $\mathcal U$, if necessary, 
$\lambda^c(F,\mu^F)>0$. The proof of Theorem~\ref{mainteo:C} is therefore finished.
\endproof

\vspace{0.5 cm}
\noindent {\bfseries Acknowledgments} We thank the anonymous referees and the editor for their thorough review and greatly appreciate their comments and suggestions, which significantly contributed to improving the quality of this work.


\bibliographystyle{plain}

\bibliography{KAN2}
\end{document}